\documentclass[reqno]{amsart}

\usepackage{hyperref}
\usepackage{amssymb}
\usepackage{dsfont}

\newcommand{\interior}[1]{\ensuremath{{\rm int} #1}}

\newcommand{\intn}[2]{\ensuremath{\{  #1  ,\ldots,  #2 \}}}

\def\grint{\mathrm{G}}
\def\grintfalc{\mathcal{G}}

\def\R{\mathbb{R}}
\def\Q{\mathbb{Q}}
\def\N{\mathbb{N}}
\def\Z{\mathbb{Z}}
\def\T{\mathbb{T}}
\def\cerc{\mathbb{S}}

\def\Fcal{\mathcal{F}}

\def\Id{\mathrm{Id}}

\def\ind{\mathds{1}}

\def\hau{\mathcal{H}}
\def\leb{\mathcal{L}}

\def\S{\mathcal{S}}

\def\netm{\mathcal{M}}
\newcommand{\diam}[1]{\ensuremath{| #1 |}}
\newcommand{\gene}[1]{\ensuremath{\langle #1 \rangle}}

\def\eps{\varepsilon}
\def\ph{\varphi}

\def\tor{\mathcal{T}}

\def\gauge{\mathfrak{D}}

\def\dist{\mathrm{d}}
\def\dd{\mathrm{d}}

\theoremstyle{plain}
\newtheorem{thm}{Theorem}
\newtheorem{prp}{Proposition}
\newtheorem{lem}[prp]{Lemma}
\newtheorem{cor}[prp]{Corollary}

\theoremstyle{definition}
\newtheorem{df}{Definition}
\newtheorem*{nota}{Notation}

\theoremstyle{remark}
\newtheorem{rem}{Remark}

\title[Large intersection properties]{Large intersection properties in Diophantine approximation and dynamical systems}
\author{Arnaud Durand}
\address{California Institute of Technology, 1200 E. California Blvd. -- MC 217-50, Pasadena, CA 91125, USA.}
\email{durand@acm.caltech.edu}

\begin{document}

\begin{abstract}
We investigate the large intersection properties of the set of points that are approximated at a certain rate by a family of affine subspaces. We then apply our results to various sets arising in the metric theory of Diophantine approximation, in the study of the homeomorphisms of the circle and in the perturbation theory for Hamiltonian systems.
\end{abstract}

\maketitle

\section{Introduction}\label{introduction}

The classical metric theory of Diophantine approximation is concerned with the description of the size properties of various sets which are typically of the form
\begin{equation}\label{defFxiri}
\Fcal_{(x_{i},r_{i})_{i\in I}}=\left\{ x\in\R^d \:\bigl|\: \|x-x_{i}\|<r_{i}\text{ for infinitely many }i\in I \right\},
\end{equation}
where $(x_{i},r_{i})_{i\in I}$ is a family of elements of $\R^d\times (0,\infty)$ indexed by some denumerable set $I$. As an illustration, let us consider one of the simplest examples of sets of the form~(\ref{defFxiri}) arising in Diophantine approximation, namely, the set
\begin{equation}\label{defKphi}
K_{\phi}=\left\{ x\in\R \:\Biggl|\: \left|x-\frac{p}{q}\right|<\phi(q)\mbox{ for infinitely many }(p,q)\in\Z\times\N \right\}
\end{equation}
formed by the reals that are $\phi$-approximable by rationals, where $\phi=(\phi(q))_{q\geq 1}$ is a nonincreasing sequence of positive real numbers converging to zero. A first description of the size properties of $K_{\phi}$ was given by Khintchine~\cite{Khintchine:1926uq}, who established that this set has full (resp.~zero) Lebesgue measure in $\R^d$ if $\sum_{q} \phi(q)q=\infty$ (resp. $<\infty$). In order to refine this description, Jarn\'ik~\cite{Jarnik:1931qf} then determined the value of the Hausdorff $g$-measure (see Section~\ref{largeint} for the definition) of $K_{\phi}$ for any {\em gauge function} $g$ in the set $\gauge_{1}$ defined as follows.

\begin{nota}
For any integer $d\geq 1$, let $\gauge_{d}$ be the set of all functions which vanish at zero, are continuous and nondecreasing on $[0,\eps]$ and are such that $r\mapsto h(r)/r^d$ is positive and nonincreasing on $(0,\eps]$, for some $\eps>0$. Moreover, for any $g,h\in\gauge_{d}$, let us write $g\prec h$ if $g/h$ monotonically tends to infinity at zero.
\end{nota}

The result of Jarn\'ik, recently improved by V.~Beresnevich, D.~Dickinson and S.~Velani~\cite{Beresnevich:2006ve}, is the following: for every gauge function $g\in\gauge_{1}$ such that $g\prec\Id$ (where $\Id$ stands for the identity function), the set $K_{\phi}$ has infinite (resp. zero) Hausdorff $g$-measure if $\sum_{q} g(\phi(q))q=\infty$ (resp. $<\infty$). On top of that, we established in~\cite{Durand:2007uq} that the set $K_{\phi}$ enjoys a remarkable property originally discovered by K.~Falconer~\cite{Falconer:1994hx}, viz., it is a {\em set with large intersection}. More precisely, for any gauge function $g\in\gauge_{1}$ enjoying $\sum_{q} g(\phi(q))q=\infty$, the set $K_{\phi}$ belongs to a certain class $\grint^g(\R)$ that we introduced in~\cite{Durand:2007uq} in order to generalize the original classes of sets with large intersection of Falconer. The class $\grint^g(\R)$ is closed under countable intersections and each of its members has infinite Hausdorff $\overline{g}$-measure in every nonempty open subset of $\R$, for every gauge function $\overline{g}\in\gauge_{1}$ enjoying $\overline{g}\prec g$. In particular, the set $K_{\phi}$ is locally everywhere of the same size, in the sense that for any gauge function $g\in\gauge_{1}$, the value of the Hausdorff $g$-measure of $K_{\phi}\cap V$ does not depend on the choice of the nonempty open subset $V$ of $\R$. This also implies that the size properties of set $K_{\phi}$ are not altered by taking countable intersections. Indeed, the Hausdorff dimension of the intersection of countably many sets with large intersection is equal to the infimum of their Hausdorff dimensions. Note that this feature is rather counterintuitive, in view of the fact that the intersection of two subsets of $\R$ of Hausdorff dimensions $s_{1}$ and $s_{2}$ respectively is usually expected to be $s_{1}+s_{2}-1$, see~\cite[Chapter~8]{Falconer:2003oj} for precise statements. We refer to Section~\ref{largeint} for more details about the classes of sets with large intersection.

As we shall show in Section~\ref{applicDyn}, Diophantine conditions, and therefore sets resembling $K_{\phi}$, arise at various points in the theory of dynamical systems and large intersection properties are particularly convenient in that context. For example, the existence of a smooth conjugacy between an orientation preserving diffeomorphism $f$ of the circle and a rotation is related with the fact that the rotation number of $f$, denoted by $\rho(f)$, is of Diophantine type $(K,\sigma)$ for some $K,\sigma>0$, which means that $|\rho(f)-p/q|\geq K/q^{\sigma+2}$ for all $p\in\Z$ and all $q\in\N$, see Subsection~\ref{rotnum}. For every $\sigma>0$, the set $\mathfrak{L}_{\sigma}$ of all reals that are not of Diophantine type $(K,\sigma)$ for any $K>0$, and thus for which the smoothness results fail, may be written as the intersection over $j\in\N$ of the sets
\begin{equation}\label{Lsigmajintro}
\left\{ \rho\in\R \:\Biggl|\: \left|\rho-\frac{p}{q}\right|<\frac{1}{j\, q^{\sigma+2}}\text{ for infinitely many }(p,q)\in\Z\times\N \right\}.
\end{equation}
Observe that each of these sets may be obtained by choosing $\phi(q)=1/(j\,q^{\sigma+2})$ in the definition~(\ref{defKphi}) of $K_{\phi}$. Hence, it belongs to the class $\grint^g(\R)$ for any gauge function $g\in\gauge_{1}$ such that the series $\sum_{q} g(q^{-\sigma-2})q$ diverges. This class being closed under countable intersections, it necessarily contains the set $\mathfrak{L}_{\sigma}$. It follows that this set has infinite Hausdorff $g$-measure in any nonempty open subset of $\R$ for any gauge function $g\in\gauge_{1}$ such that $\sum_{q} g(q^{-\sigma-2})q=\infty$ (see the proof of Theorem~\ref{thmLsigma} for details). The classes $\grint^g(\R)$ make the proof of this result particularly straightforward, because of their stability under countable intersections and the fact that $\mathfrak{L}_{\sigma}$ is the countable intersection of the sets defined by~(\ref{Lsigmajintro}). Also, the fact that the set $\mathfrak{L}_{\sigma}$ is a set with large intersection implies that the rotation numbers for which the smoothness results fail are ``omnipresent'' in $\R$ in a very strong measure theoretic sense.

The description of the size and large intersection properties of the set $K_{\phi}$ that we briefly presented above follows from very general methods concerning the set $\Fcal_{(x_{i},r_{i})_{i\in I}}$ defined by~(\ref{defFxiri}). By covering $\Fcal_{(x_{i},r_{i})_{i\in I}}$ by an appropriate union of balls with centers $x_{i}$ and radii $r_{i}$, it is usually obvious to provide a sufficient condition on the family $(x_{i},r_{i})_{i\in I}$ to ensure that this set has Lebesgue measure zero or a sufficient condition on the family $(x_{i},r_{i})_{i\in I}$ and the gauge function $g$ to establish that the set has Hausdorff $g$-measure zero. Conversely, it is usually much more awkward to provide a sufficient condition to ensure that $\Fcal_{(x_{i},r_{i})_{i\in I}}$ has full Lebesgue measure or has infinite $g$-measure. The most recent results on that question were obtained by Beresnevich, Dickinson and Velani~\cite{Beresnevich:2006ve}, who basically solved the problem in the case where the family $(x_{i},r_{i})_{i\in I}$ leads to what they call a {\em ubiquitous system}. Moreover, Beresnevich and Velani~\cite{Beresnevich:2005vn} proved the following {\em mass transference principle}: for any gauge function $g\in\gauge_{d}$ and any nonempty open subset $V$ of $\R^d$ such that the set $\Fcal_{(x_{i},g(r_{i})^{1/d})_{i\in I}}$ has full Lebesgue measure in $V$, the set $\Fcal_{(x_{i},r_{i})_{i\in I}}$ has maximal Hausdorff $g$-measure in $V$. Thus, together with the mass transference principle, the sole knowledge of Lebesgue measure theoretic statements for a set of the form~(\ref{defFxiri}) yields a complete description of its size properties. Under the same hypotheses, we established in~\cite{Durand:2007uq} that $\Fcal_{(x_{i},r_{i})_{i\in I}}$ belongs to the class $\grint^g(V)$ of sets with large intersection in $V$ with respect to the gauge function $g$, see Section~\ref{ubiquitysec} for details. We successfully used this result to completely describe the large intersection properties of various sets of the form~(\ref{defFxiri}) arising in metric number theory~\cite{Durand:2007uq}, such as $K_{\phi}$, or coming into play in the multifractal analysis of a L\'evy process~\cite{Durand:2007fk} and a new model of random wavelet series with correlated coefficients~\cite{Durand:2007kx}.

Subsequently, Beresnevich and Velani~\cite{Beresnevich:2006lr} observed that their mass transference principle could be extended to the more general situation in which the set $\Fcal_{(x_{i},r_{i})_{i\in I}}$ is replaced by the set
\begin{equation}\label{defFPiri}
\Fcal_{(P_{i},r_{i})_{i\in I}}=\left\{ x\in\R^d \:\bigl|\: \dist(x,P_{i})<r_{i} \text{ for infinitely many } i\in I \right\},
\end{equation}
formed by the points in $\R^d$ that are at a distance less than $r_{i}$ of a given affine subspace $P_{i}$ for infinitely many indices $i\in I$. Using this observation, they investigated the size properties of the generalization of $K_{\phi}$ to the linear forms setting, thereby complementing Lebesgue measure theoretic results obtained by W.~Schmidt~\cite{Schmidt:1964vn}. In this paper, we show that, under simple assumptions bearing on the affine subspaces $P_{i}$ and the radii $r_{i}$, the set $\Fcal_{(P_{i},r_{i})_{i\in I}}$ is a set with large intersection, in the sense that it belongs to some of the aforementioned classes $\grint^g(V)$, see Section~\ref{ubiquitysec}. This way, we are able to investigate the large intersection properties of the set studied by Schmidt, Beresnevich and Velani, see Section~\ref{applicDioph}.

Our approach also enables us to describe the size and large intersection properties of various sets arising in the Kolmogorov-Arnold-Moser theory on the perturbations of a Hamiltonian system, see Subsection~\ref{applicKAM}. In particular, we prove that the set of frequencies for which the constructions involved in this theory fail is a set with large intersection. This implies that those ``problematic'' frequencies are omnipresent in a strong measure theoretic sense. As in the study of the homeomorphisms of the circle, the fact that the classes $\grint^g(V)$ are closed under countable intersections is particularly convenient in that context.

The paper is organized as follows. In Section~\ref{largeint}, we recall the definition of Hausdorff measures and we give a brief overview of the classes of sets with large intersection introduced in~\cite{Durand:2007uq}. We present in Section~\ref{ubiquitysec} the main result of the paper, according to which the set $\Fcal_{(P_{i},r_{i})_{i\in I}}$ is a set with large intersection. In Section~\ref{applicDioph}, we then apply our results to the study of the large intersection properties of the set arising in the linear forms setting in Diophantine approximation. Applications to the theory of dynamical systems are discussed in Section~\ref{applicDyn}. Specifically, we describe the size and large intersection properties of various sets appearing in the study of the perturbations of Hamiltonian systems and the homeomorphisms of the circle. Lastly, the proofs of the main results of the paper are given in Section~\ref{proofubiquity2} and~\ref{proofthmRnnu}.

\section{Hausdorff measures and large intersection properties}\label{largeint}

Before discussing large intersection properties, let us recall some definitions and basic results about Hausdorff measures. Let $\gauge$ be the set of all nondecreasing functions $g$ defined on $[0,\eps]$ for some $\eps>0$ and such that $\lim_{0^+} g=g(0)=0$. For any gauge function $g\in\gauge$, the Hausdorff $g$-measure of a set $F\subseteq\R^d$ is defined by
\[
\hau^g(F)=\lim_{\delta\downarrow 0}\uparrow\hau^g_{\delta}(F) \qquad\text{with}\qquad \hau^g_{\delta}(F)=\inf_{F\subseteq\bigcup_{p} U_{p}\atop\diam{U_{p}}<\delta} \sum_{p=1}^\infty g(\diam{U_{p}}).
\]
The infimum is taken over all sequences $(U_{p})_{p\geq 1}$ of sets with $F\subseteq\bigcup_{p} U_{p}$ and $\diam{U_{p}}<\delta$ for all $p\geq 1$, where $\diam{\cdot}$ denotes diameter. Note that $\hau^g$ is a Borel 
measure on $\R^d$, see~\cite{Rogers:1970wb}. Actually, in view of the following result of~\cite{Durand:2007uq} and given that the sets that we study hereunder have Hausdorff measure either zero or infinity (see Sections~\ref{applicDioph} and~\ref{applicDyn}), we shall restrict our attention to gauges in the set $\gauge_{d}$ defined in Section~\ref{introduction}.

\begin{prp}\label{prpgauged}
For every gauge function $g\in\gauge$, the function
\[
g_{d}:r\mapsto r^d\inf_{\rho\in (0,r]}\frac{g(\rho)}{\rho^d}.
\]
either belongs to $\gauge_{d}$ or is equal to zero near zero. Moreover, there is a real number $\kappa\geq 1$ such that for every $g\in\gauge$ and every $F\subseteq\R^d$,
\[
\hau^{g_{d}}(F)\leq\hau^g(F)\leq\kappa\,\hau^{g_{d}}(F).
\]
\end{prp}

Observe that for $g\in\gauge_{d}$, if $g\prec\Id^d$, every nonempty open subset of $\R^d$ has infinite Hausdorff $g$-measure. Otherwise, $g(r)={\rm O}(r^d)$ as $r$ goes to zero, so that $\hau^g$ is finite on every compact subset of $\R^d$. Since it is a translation invariant Borel measure, it coincides up to a multiplicative constant with the Lebesgue measure on the Borel subsets of $\R^d$.

Lastly, recall that the Hausdorff dimension of a nonempty set $F\subseteq\R^d$ is defined with the help of the gauge functions $\Id^s$ by
\[
\dim F=\sup\{s\in (0,d) \:|\: \hau^{\Id^s}(F)=\infty\}=\inf\{s\in (0,d) \:|\: \hau^{\Id^s}(F)=0\}
\]
with the convention that $\sup\emptyset=0$ and $\inf\emptyset=d$, see~\cite{Falconer:2003oj}.

\medskip

In~\cite{Durand:2007uq}, we introduced new classes of sets with large intersection which generalize the classes $\grintfalc^s(\R^d)$ originally considered by Falconer~\cite{Falconer:1994hx}. In the remainder of this section, we give a brief overview of these new classes and we refer to~\cite{Durand:2007uq} for a fuller exposition. Our classes are associated with the functions that belong to the set $\gauge_{d}$ defined in Section~\ref{introduction} and are obtained in the following manner, with the help of outer net measures. Given an integer $c\geq 2$, let $\Lambda_{c}$ be the collection of the $c$-adic cubes of $\R^d$, that is, the sets of the form $\lambda=c^{-j}( k+[0,1)^d )$ for $j\in\Z$ and $k\in\Z^d$. The integer $j$ is the generation of $\lambda$, denoted by $\gene{\lambda}_{c}$. For any $g\in\gauge_{d}$, the set of all $\eps\in (0,1]$ such that $g$ is nondecreasing on $[0,\eps]$ and $r\mapsto g(r)/r^d$ is nonincreasing on $(0,\eps]$ is nonempty. Let $\eps_{g}$ denote its supremum. The outer net measure associated with $g\in\gauge_{d}$ is defined by
\begin{equation}\label{defnetmh}
\forall F\subseteq\R^d \qquad \netm^g_{\infty}(F) = \inf_{(\lambda_{p})_{p\geq 1}} \sum_{p=1}^\infty g(\diam{\lambda_{p}}),
\end{equation}
where the infimum is taken over all sequences $(\lambda_{p})_{p\geq 1}$ with $F\subseteq\bigcup_{p}\lambda_{p}$, where each $\lambda_{p}$ is either a cube in $\Lambda_{c}$ with diameter less than $\eps_{g}$ or the empty set. As shown by~\cite[Theorem~49]{Rogers:1970wb}, the outer measure $\netm^g_{\infty}$ is in some way related with the Hausdorff measure $\hau^g$. In particular, if a subset $F$ of $\R^d$ enjoys $\netm^g_{\infty}(F)>0$, then $\hau^g(F)>0$. The classes of sets with large intersection introduced in~\cite{Durand:2007uq} are now defined as follows. Recall that a $G_{\delta}$-set is one that may be expressed as a countable intersection of open sets.

\begin{df}
Let $g\in\gauge_{d}$ and let $V$ be a nonempty open subset of $\R^d$. The class $\grint^g(V)$ of subsets of $\R^d$ with large intersection in $V$ with respect to $g$ is the collection of all $G_{\delta}$-subsets $F$ of $
\R^d$ such that $\netm^{\overline{g}}_{\infty}(F\cap U)=\netm^{\overline{g}}_{\infty}(U)$ for every $\overline{g}\in\gauge_{d}$ enjoying $\overline{g}\prec g$ and every open set $U\subseteq V$.
\end{df}

\begin{rem}
The classes $\grint^g(V)$ depend on the choice of neither the integer $c$ nor the norm $\R^d$ is endowed with, even if they affect the construction of $\netm^{\overline{g}}_{\infty}$ for any $\overline{g}\in\gauge_{d}$ with $\overline{g}\prec g$, see~\cite[Proposition~13]{Durand:2007uq}.
\end{rem}

The next proposition gives the basic properties of the classes $\grint^g(V)$ that follow directly from their definition.

\begin{prp}\label{grintpropbase}
Let $g\in\gauge_{d}$ and let $V$ be a nonempty open subset of $\R^d$. Then
\renewcommand{\theenumi}{\alph{enumi}}
\begin{enumerate}
\item $\grint^{g_{1}}(V) \supseteq \grint^{g_{2}}(V)$ for any $g_{1},g_{2}\in\gauge_{d}$ with $g_{1}\prec g_{2}$;
\item $\grint^g(V_{1}) \supseteq \grint^g(V_{2})$ for any nonempty open sets $V_{1},V_{2}\subseteq\R^d$ with $V_{1}\subseteq V_{2}$;
\item $\grint^g(V)=\bigcap_{\overline{g}}\grint^{\overline{g}}(V)$ where $\overline{g}\in\gauge_{d}$ enjoys $\overline{g}\prec g$;
\item $\grint^g(V)=\bigcap_{U}\grint^g(U)$ where $U$ is a nonempty open subset of $V$;
\item\label{propbaseextension} every $G_{\delta}$-set which contains a set of $\grint^g(V)$ also belongs to $\grint^g(V)$;
\item $F\cap U\in\grint^g(U)$ for every $F\in\grint^g(V)$ and every nonempty open set $U\subseteq V$.
\end{enumerate}
\renewcommand{\theenumi}{\roman{enumi}}
\end{prp}

The following result, which combines Theorem~1 and Proposition~11 in~\cite{Durand:2007uq}, provides the main nontrivial properties of the classes $\grint^g(V)$. These properties show in particular that a set with large intersection in some nonempty open set $V$ is to be thought of as large and omnipresent in $V$, in a measure theoretic sense.

\begin{thm}\label{grintstable}
Let $g\in\gauge_{d}$ and let $V$ be a nonempty open subset of $\R^d$. Then,
\renewcommand{\theenumi}{\alph{enumi}}
\begin{enumerate}
\item\label{intersect} the class $\grint^g(V)$ is closed under countable intersections;
\item\label{bilip} the set $f^{-1}(F)$ belongs to $\grint^g(V)$ for every bi-Lipschitz mapping $f:V\to\R^d$ and every set $F\in\grint^g(f(V))$;
\item\label{relsizelargeint} every set $F\in\grint^g(V)$ enjoys $\hau^{\overline{g}}(F)=\infty$ for every $\overline{g}\in\gauge_{d}$ with $\overline{g}\prec g$ and in particular
\[
\dim F\geq s_{g}=\sup\left\{s\in (0,d) \:|\: \Id^s\prec g \right\};
\]
\item\label{fullleb} every $G_{\delta}$-subset of $\R^d$ with full Lebesgue measure in $V$ is in the class $\grint^g(V)$.
\end{enumerate}
\renewcommand{\theenumi}{\roman{enumi}}
\end{thm}

Using Theorem~\ref{grintstable}, it is possible to establish that $\grint^g(\R^d)$ is included in the class $\grintfalc^{s_{g}}(\R^d)$ of Falconer when $s_{g}$ is positive, see~\cite{Durand:2007uq}. To end this section, let us indicate another noteworthy consequence of Theorem~\ref{grintstable}. Let $g\in\gauge_{d}$ and let $V$ be a nonempty open subset of $\R^d$. For any sequence $(F_{n})_{n\geq 1}$ of sets in the class $\grint^h(V)$,
\[
\forall \overline{g}\in\gauge_{d} \qquad \overline{g}\prec g \quad\Longrightarrow\quad \hau^{\overline{g}}\left(\bigcap_{n=1}^\infty F_{n}\right)=\infty.
\]
Hence the Hausdorff dimension of $\bigcap_{n} F_{n}$ is at least $s_{g}$. In addition, if the dimension of $F_{n}$ is at most $s_{g}$ for some $n\geq 1$, the previous intersection has dimension $s_{g}$.

\section{Approximation by affine subspaces}\label{ubiquitysec}

Let $I$ denote a denumerable set and let $\S_{d}(I)$ be the set of all families $(x_{i},r_{i})_{i\in I}$ of elements of $\R^d\times (0,\infty)$ such that
\[
\sup_{i\in I} r_{i}<\infty \qquad\mbox{and}\qquad \forall m\in\N \quad \#\left\{ i\in I \:\bigl|\: \|x_{i}\|<m\mbox{ and }r_{i}>1/m \right\}<\infty.
\]
The set of all points in $\R^d$ that are infinitely often at a distance less than $r_{i}$ of the point $x_{i}$ is given by~(\ref{defFxiri}), that is,
\[
\Fcal_{(x_{i},r_{i})_{i\in I}}=\left\{ x\in\R^d \:\bigl|\: \|x-x_{i}\|<r_{i}\text{ for infinitely many }i\in I \right\}.
\]
Sets of this form play a central role in various areas of mathematics, such as number theory and multifractal analysis, see for instance~\cite{Beresnevich:2006ve,Durand:2007kx,Durand:2007uq,Durand:2007fk,Durand:2006uq}. Several examples arising in Diophantine approximation are mentioned in the beginning of Section~\ref{applicDioph}. In multifractal analysis, sets of the form~(\ref{defFxiri}) are obtained by considering the points at which a stochastic process, such as a L\'evy process or a random wavelet series, has at most a given H\"older exponent.

We established in~\cite{Durand:2007uq} that, under a very general assumption on the family $(x_{i},r_{i})_{i\in I}$, the set $\Fcal_{(x_{i},r_{i})_{i\in I}}$ is a set with large intersection with respect to a given gauge function $h
\in\gauge_{d}$. To be specific, Theorem~2 in~\cite{Durand:2007uq} straightforwardly implies the following result.

\begin{thm}\label{ubiquity}
Let $I$ be a denumerable set, let $(x_{i},r_{i})_{i\in I}\in\S_{d}(I)$, let $h\in\gauge_{d}$ and let $V$ be a nonempty open subset of $\R^d$. Assume that for Lebesgue-almost every $x\in V$, there exist infinitely many $i\in I
$ such that
\[
\|x-x_{i}\|<h(r_{i})^{1/d}.
\]
Then, the set $\Fcal_{(x_{i},r_{i})_{i\in I}}$ defined by~(\ref{defFxiri}) belongs to the class $\grint^h(V)$.
\end{thm}

\begin{rem}\label{masstrprp}
In view of the relationship between size and large intersection properties given by Theorem~\ref{grintstable}(\ref{relsizelargeint}), Theorem~\ref{ubiquity} is to be compared with the mass transference principle 
established by Beresnevich and Velani in~\cite{Beresnevich:2005vn}, which states that, under the same assumptions, the set $\Fcal_{(x_{i},r_{i})_{i\in I}}$ has maximal Hausdorff $h$-measure in every open subset of 
$V$. Nevertheless, none of these results implies the other one.
\end{rem}

\medskip

In fact, we adopted in~\cite{Durand:2007uq} a slightly more general approach which relies on the notion of {\em homogeneous ubiquitous system} that is defined as follows.

\begin{df}\label{ubsys}
Let $I$ be a denumerable set and let $V$ be a nonempty open subset of $\R^d$. A family $(x_{i},r_{i})_{i\in I}\in\S_{d}(I)$ is called a homogeneous ubiquitous system in $V$ if the set $\Fcal_{(x_{i},r_{i})_{i\in I}}$ given by~(\ref{defFxiri}) has full Lebesgue measure in $V$.
\end{df}

\begin{rem}\label{remubsyscst}
By virtue of~\cite[Proposition~15]{Durand:2007uq}, if $(x_{i},r_{i})_{i\in I}\in\S_{d}(I)$ is a homogeneous ubiquitous system in $V$, so is $(x_{i},\kappa r_{i})_{i\in I}$ for any $\kappa>0$. Thus, the fact that $(x_{i},r_{i})_{i\in I}\in\S_{d}(I)$ is a homogeneous ubiquitous system in $V$ does not depend on the choice of the norm $\R^d$ is endowed with.
\end{rem}

As an example, for any integer $c\geq 2$, the family $(kc^{-j},c^{-j})_{(j,k)\in\N\times\Z^d}$ is a homogeneous ubiquitous system in $\R^d$. Similarly, Dirichlet's theorem ensures that for any $x\in\R^d$, there are infinitely many $(p,q)\in\Z^d\times\N$ such that $\|x-p/q\|_{\infty} < q^{-1-1/d}$, where $\|\cdot\|_{\infty}$ denotes the supremum norm, see~\cite[Theorem~200]{Hardy:1979fk}. Hence, $(p/q,q^{-1-1/d})_{(p,q)\in\Z^d\times\N}$ is a homogeneous ubiquitous system in $\R^d$. In addition, the optimal regular systems of points defined in~\cite{Baker:1970jf,Beresnevich:2000fk} also yield homogeneous ubiquitous systems. Examples of regular systems include the points with rational coordinates, the real algebraic numbers of bounded degree and the algebraic integers of bounded degree, see~\cite{Beresnevich:1999ys,Beresnevich:2002kx,Bugeaud:2002fk,Bugeaud:2002uq,Bugeaud:2004zr}. We refer to~\cite{Durand:2007uq} for details.

Now, given a gauge function $h\in\gauge_{d}$, the pseudo-inverse function of $h^{1/d}$ is defined on the interval $[0,h^{1/d}({\eps_{h}}^-))$ by
\[
(h^{1/d})^{-1}:r\mapsto\inf\{\rho\in [0,\eps_{h}) \:|\: h^{1/d}(\rho)\geq r\},
\]
where $h^{1/d}({\eps_{h}}^-)$ is equal to $\sup_{[0,\eps_{h})}h^{1/d}>0$. Theorem~2 in~\cite{Durand:2007uq}, which leads to Theorem~\ref{ubiquity} above, is stated as follows.

\begin{thm}\label{ubiquityold}
Let $I$ be a denumerable set, let $V$ be a nonempty open subset of $\R^d$ and let $(x_{i},r_{i})_{i\in I}\in\S_{d}(I)$ be a homogeneous ubiquitous system in $V$. Then, for any gauge function $h\in\gauge_{d}$ and any nonnegative nondecreasing function $\ph:[0,\infty)\to\R$ coinciding with $(h^{1/d})^{-1}$ near zero, the set $\Fcal_{(x_{i},\ph(r_{i}))_{i\in I}}$ belongs to the class $\grint^h(V)$.
\end{thm}

\medskip

Recall that the set $\Fcal_{(x_{i},r_{i})_{i\in I}}$ defined by~(\ref{defFxiri}) is composed by the points in $\R^d$ that are at a distance less than $r_{i}$ of the point $x_{i}$ for infinitely many $i\in I$. Hence, a natural generalization of $\Fcal_{(x_{i},r_{i})_{i\in I}}$ is the set of points in $\R^d$ that are at a distance less than $r_{i}$ of some affine subspace $P_{i}$ for infinitely many $i\in I$. Specifically, let $k\in\intn{0}{d-1}$, let $I$ be a denumerable set 
and let $\S_{d}^k(I)$ be the set of all families $(P_{i},r_{i})_{i\in I}$ formed by affine subspaces $P_{i}$ of $\R^d$ with dimension $k$ and positive reals $r_{i}$ such that
\begin{equation}\label{condPiri}
\sup_{i\in I}r_{i}<\infty \qquad\text{and}\qquad \forall m\in\N \qquad \#\left\{ i\in I \:\biggl| \begin{array}{l} P_{i}\cap B_{m}\neq\emptyset \\ \mbox{and }r_{i}>1/m\end{array} \right\}<\infty,
\end{equation}
where $B_{m}$ denotes the open ball with center zero and radius $m$. Note that, identifying a point with the zero-dimensional affine subspace that contains it, we may write $\S_{d}^0(I)=\S_{d}(I)$. The natural extension of 
the set defined by~(\ref{defFxiri}) is then the set defined by~(\ref{defFPiri}), namely,
\[
\Fcal_{(P_{i},r_{i})_{i\in I}}=\left\{ x\in\R^d \:\bigl|\: \dist(x,P_{i})<r_{i} \text{ for infinitely many } i\in I \right\}.
\]

Theorem~\ref{ubiquity2} below shows that, under certain assumptions on the subspaces $P_{i}$ and the radii $r_{i}$, the set $\Fcal_{(P_{i},r_{i})_{i\in I}}$ is a set with large intersection. This result, which may thus be seen as 
the extension of Theorem~\ref{ubiquity} to $\Fcal_{(P_{i},r_{i})_{i\in I}}$, is proven in Section~\ref{proofubiquity2}.

\begin{thm}\label{ubiquity2}
Let $k\in\intn{0}{d-1}$, let $I$ be a denumerable set, let $(P_{i},r_{i})_{i\in I}\in\S_{d}^k(I)$, let $h\in\gauge_{d-k}$ and let $V$ be a nonempty open subset of $\R^d$. Assume that:
\renewcommand{\theenumi}{\Alph{enumi}}
\begin{enumerate}
\item\label{existT} there exists an affine subspace $T$ of $\R^d$ with dimension $d-k$ such that $T\cap P_{i}\neq\emptyset$ for all $i\in I$ and
\[
C=\sup_{i\in I}\bigl|\{x\in T \:|\: \dist(x,P_{i})<1\}\bigr|<\infty;
\]
\item\label{existhunderline} there exists a gauge function $\underline{h}\in\gauge_{d-k}$ with $h\prec\underline{h}$ such that for Lebesgue-almost every $x\in V$, there are infinitely many indices $i\in I$ enjoying
\begin{equation}\label{ubiquity2eq}
\dist(x,P_{i})<\underline{h}(r_{i})^{\frac{1}{d-k}}.
\end{equation}
\end{enumerate}
\renewcommand{\theenumi}{\roman{enumi}}
Then, the set $\Fcal_{(P_{i},r_{i})_{i\in I}}$ belongs to the class $\grint^{\Id^k h}(V)$.
\end{thm}

\begin{rem}\label{ubiquityBV}
Assume that~(\ref{existT}) and~(\ref{existhunderline}) hold, let $\tilde h=\sqrt{h\underline{h}}$ and observe that $\tilde h\in\gauge_{d-k}$ and $h\prec\tilde h\prec\underline{h}$. Applying Theorem~\ref{ubiquity2} with $
\tilde h$ instead of $h$ leads to the fact that $\Fcal_{(P_{i},r_{i})_{i\in I}}\in\grint^{\Id^k \tilde h}(V)$. Theorem~\ref{grintstable}(\ref{relsizelargeint}) then implies that
\[
\hau^{\Id^k h}(\Fcal_{(P_{i},r_{i})_{i\in I}}\cap U)=\hau^{\Id^k h}(U)
\]
for every open subset $U$ of $V$. Beresnevich and Velani~\cite{Beresnevich:2006lr} previously obtained the same result, when~(\ref{existhunderline}) is replaced by the weaker assumption that for Lebesgue-almost every $x\in V$,
\begin{equation}\label{ubiquityBVeq}
\dist(x,P_{i})<h(r_{i})^{\frac{1}{d-k}} \qquad \text{for infinitely many}\ i\in I
\end{equation}
and when for any $\eps>0$, only finitely many $i\in I$ enjoy $r_{i}>\eps$. This result may be regarded as an extension of the mass transference principle mentioned in Section~\ref{introduction} and Remark~\ref{masstrprp}.
\end{rem}

\begin{rem}
Under the weaker assumption that~(\ref{ubiquityBVeq}) holds for Lebesgue-almost every $x\in V$, the proof of Theorem~\ref{ubiquity2} entails that $\netm^{\Id^k\overline{h}}_{\infty}(\Fcal_{(P_{i},r_{i})_{i\in I}}\cap U)=\netm^{\Id^k\overline{h}}_{\infty}(U)$ for any gauge function $\overline{h}\in\gauge_{d-k}$ such that $\overline{h}\prec h$ and any open set $U\subseteq V$, see Section~\ref{proofubiquity2}. This result is weaker than the fact that $\Fcal_{(P_{i},r_{i})_{i\in I}}\in\grint^{\Id^k h}(V)$, because the gauge functions $g\in\gauge_{d}$ for which $g\prec\Id^k h$ are not necessarily of the form $\Id^k\overline{h}$ with $\overline{h}\in\gauge_{d-k}$ and $\overline{h}\prec h$.
\end{rem}

\section{Applications to Diophantine approximation}\label{applicDioph}

In~\cite{Durand:2007uq}, we made use of Theorems~\ref{grintstable} and~\ref{ubiquity} in order to provide a full description of the size and large intersection properties of various sets arising in classical Diophantine 
approximation, such as the set of all points that are approximable with a certain accuracy by rationals, by rationals with restricted numerator and denominator or by real algebraic numbers. For example, for any 
nonincreasing sequence $\phi=(\phi(q))_{q\geq 1}$ of positive real numbers converging to zero, we employed Theorems~\ref{grintstable} and~\ref{ubiquity} in order to study the size and large intersection properties of the 
set
\begin{equation}\label{defKmphi}
K_{m,\phi}=\left\{ x\in\R^m \:\Biggl|\: \left\|x-\frac{p}{q}\right\|<\phi(q)\mbox{ for infinitely many }(p,q)\in\Z^m\times\N \right\}.
\end{equation}
This set was first studied by Khintchine~\cite{Khintchine:1926uq} in 1926 and, for $m=1$, it is equal to the set $K_{\phi}$ defined by~(\ref{defKphi}). Note that it is of the form~(\ref{defFxiri}) and is composed by the points $x\in\R^m$ (with $m\in\N$) enjoying $|qx|_{\Z^m}<q\phi(q)$ for 
infinitely many integers $q\in\N$, where $|y|_{\Z^m}=\min_{k\in\Z^m}\|y-k\|$ denotes the distance from a given point $y\in\R^m$ to $\Z^m$. The result of~\cite{Durand:2007uq} describing the size and 
large intersection properties of $K_{m,\phi}$ is the following.

\begin{thm}\label{grintkhintchine}
Let $\phi=(\phi(q))_{q\geq 1}$ be a nonincreasing sequence of positive real numbers converging to zero, let $h\in\gauge_{m}$ and let $V$ be a nonempty open subset of $\R^m$. Then,
\[
\left\{\begin{array}{lcl}
\sum_{q} h(\phi(q))q^m=\infty & \qquad\Longrightarrow\qquad & \hau^h(K_{m,\phi}\cap V)=\hau^h(V) \\[2mm]
\sum_{q} h(\phi(q))q^m<\infty & \qquad\Longrightarrow\qquad & \hau^h(K_{m,\phi}\cap V)=0.
\end{array}\right.
\]
Moreover,
\[
K_{m,\phi}\in\grint^h(V) \qquad\Longleftrightarrow\qquad \sum\nolimits_{q} h(\phi(q)) q^m=\infty.
\]
\end{thm}

The purpose of this section is to establish the same kind of result for a more general set which involves linear forms and is defined in the following manner. Let $\Psi_{n}$ denote the set of all nonnegative functions $\psi$ defined on $\Z^n$ (with $n\in\N$) such that $\psi(q)$ tends to zero as $\|q\|$ tends to infinity. For any function $\psi\in\Psi_{n}$ and any point $b\in\R^m$, let us consider the set
\begin{equation}\label{defSbmnpsi}
S^b_{m,n,\psi}=\left\{ (x_{1},\ldots,x_{m})\in (\R^n)^m \:\biggl|\: \begin{array}{ll}\sup\limits_{1\leq j\leq m} |q\cdot x_{j}-b_{j}|_{\Z}<\psi(q)\\ \text{for infinitely many } q\in\Z^n\end{array}\right\},
\end{equation}
where $\,\cdot\,$ denotes the standard inner product in $\R^n$. Of course, the set $S^b_{m,n,\psi}$ may be regarded as a subset of $\R^{mn}$. Moreover, it is easy to check that the set $K_{m,\phi}$ defined by~(\ref{defKmphi}) can be obtained from the set $S^b_{m,n,\psi}$ by letting $b=0$, 
$n=1$ and $\phi(q)=\ind_{\{q\geq 1\}}\psi(q)/q$ for any integer $q\in\Z$.

The size properties of the set $S^b_{m,n,\psi}$ were first investigated by Schmidt, who obtained in~\cite{Schmidt:1964vn} the following result concerning its Lebesgue measure.

\begin{thm}[Schmidt]\label{thmSchmidt}
Assume that $m+n>2$. Let $b\in\R^m$ and $\psi\in\Psi_{n}$. Then, for any open subset $V$ of $\R^{mn}$,
\[\left\{\begin{array}{lcl}
\sum_{q\in\Z^n}\psi(q)^m=\infty &\:\Longrightarrow\:& \leb^{mn}(S^b_{m,n,\psi}\cap V)=\leb^{mn}(V) \\[2mm]
\sum_{q\in\Z^n}\psi(q)^m<\infty &\:\Longrightarrow\:& \leb^{mn}(S^b_{m,n,\psi}\cap V)=0.
\end{array}\right.\]
\end{thm}

In the case where $m$ and $n$ are both equal to one, the previous result does not hold and the appropriate statement would follow from the settlement of the Duffin-Schaeffer conjecture, see~\cite{Beresnevich:2005vn}.

More recently, Beresnevich and Velani~\cite{Beresnevich:2006lr} extended Theorem~\ref{thmSchmidt} to the Hausdorff measures associated with the gauge functions $\Id^{m(n-1)}h$, for $h\in\gauge_{m}$.

\begin{thm}[Beresnevich and Velani]\label{thmBV}
Assume that $m+n>2$. Let $b\in\R^m$ and $\psi\in\Psi_{n}$. Then, for any gauge function $h\in\gauge_{m}$ and any open subset $V$ of $\R^{mn}$,
\[\left\{\begin{array}{lcl}
\sum_{q\in\Z^n\setminus\{0\}} h(\frac{\psi(q)}{\|q\|})\|q\|^m=\infty &\:\Longrightarrow\:& \hau^{\Id^{m(n-1)}h}(S^b_{m,n,\psi}\cap V)=\hau^{\Id^{m(n-1)}h}(V) \\[2mm]
\sum_{q\in\Z^n\setminus\{0\}} h(\frac{\psi(q)}{\|q\|})\|q\|^m<\infty &\:\Longrightarrow\:& \hau^{\Id^{m(n-1)}h}(S^b_{m,n,\psi}\cap V)=0.
\end{array}\right.\]
\end{thm}

\begin{rem}
Note that the summability condition clearly does not depend on the choice of the norm $\R^n$ is endowed with, because $h$ belongs to $\gauge_{m}$.
\end{rem}

\begin{rem}\label{remformgauge}
It is highly probable that the statement of Theorem~\ref{thmBV} may not be extended to the gauge functions that are not of the form $\Id^{m(n-1)}h$ with $h\in\gauge_{m}$. For example, if Theorem~\ref{thmBV} held for 
the gauge $\Id^{m(n-1)}$, it would ensure that the Hausdorff $\Id^{m(n-1)}$-measure of $S^b_{m,n,\psi}$ is infinite (because the sum of $\|q\|^m$ over $q\in\Z^n\setminus\{0\}$ diverges). Nonetheless, $S^b_{m,n,\psi}
$ can be regarded as the set of all points in $\R^{mn}$ that are approximable at a certain rate by a family of $m(n-1)$-dimensional affine subspaces, see~(\ref{distPbpqpsi}) below. Therefore, when $\psi$ tends rapidly to 
zero at infinity, the $\Id^{m(n-1)}$-measure of this set could be finite, depending on some specific arithmetic properties enjoyed by the approximating subspaces. See the discussion at the end of~\cite[Section~1.2]
{Beresnevich:2006lr} for details.
\end{rem}

Recall that if the gauge function $h\in\gauge_{m}$ is such that $h\not\prec\Id^m$, the Hausdorff $\Id^{m(n-1)}h$-measure coincides, up to a multiplicative constant, with the Lebesgue measure on the Borel subsets of $\R^{mn}$. 
Thus, in this case, Theorem~\ref{thmBV} directly follows from Theorem~\ref{thmSchmidt}. In the case where $h\prec\Id^m$, the convergence part of Theorem~\ref{thmBV} may be easily proven by covering the set 
$S^b_{m,n,\psi}$ in an appropriate way. In order to establish the divergence part, Beresnevich and Velani used the result mentioned in Remark~\ref{ubiquityBV}, along with Theorem~\ref{thmSchmidt}, after observing that $S^b_{m,n,\psi}$ is of the form~(\ref{defFPiri}). Indeed, it is easy to check that $S^b_{m,n,\psi}$ is composed by the points $x\in\R^{mn}$ enjoying
\begin{equation}\label{distPbpqpsi}
\dist_{*}(x,P^b_{(p,q)})<\frac{\psi(q)}{\|q\|_{2}}
\end{equation}
for infinitely many $(p,q)\in\Z^m\times (\Z^n\setminus\{0\})$, where $\|\cdot\|_{2}$ is the Euclidean norm. Here, $\dist_{*}(x,P^b_{(p,q)})$ denotes the distance from the point $x$ to the approximating subspace
\[
P^b_{(p,q)}=\left\{ (y_{1},\ldots,y_{m})\in (\R^n)^m \:\bigl|\: \forall j\in\intn{1}{m} \quad q\cdot y_{j}=b_{j}+p_{j} \right\},
\]
when the space $\R^{mn}$ is endowed with the norm $\|\cdot\|_{*}$ defined by
\[
\forall x=(x_{1},\ldots,x_{m})\in (\R^n)^m \qquad \|x\|_{*}=\sup_{1\leq j\leq m}\|x_{j}\|_{2}.
\]
In fact, the subspaces $P^b_{(p,q)}$, for $p\in\Z^m$ and $q\in\Z^n\setminus\{0\}$, do not verify~(\ref{existT}), so Beresnevich and Velani applied the result mentioned in Remark~\ref{ubiquityBV} only to the subspaces $P^b_{(p,q)}$ for which $q$ belongs to the set
\begin{equation}\label{defQi}
\mathcal{Q}_{i}=\left\{ q=(q_{1},\ldots,q_{n})\in\Z^n\setminus\{0\} \:\bigl|\: \|q\|_{\infty}=q_{i} \right\},
\end{equation}
where $i\in\{1,\ldots,n\}$ is chosen in advance depending on the approximating function $\psi$. Those particular subspaces $P^b_{(p,q)}$ enjoy~(\ref{existT}) with common 
subspace the set $T_{i}$ of all $(x_{1,1},\ldots,x_{1,n},\ldots,x_{m,1},\ldots,x_{m,n})\in\R^{mn}$ such that $x_{j,i'}=0$ for all $j\in\intn{1}{m}$ and all $i'\in\intn{1}{n}\setminus\{i\}$. Along with the radii $r_{(p,q)}=\psi(q)/\|q\|_{2}$, they also enjoy~(\ref{condPiri}), so that the family $(P^b_{(p,q)},r_{(p,q)})_{(p,q)\in\Z^m\times\mathcal{Q}_{i}}$ belongs to the collection $\S_{mn}^{m(n-1)}(\Z^m\times\mathcal{Q}_{i})$.

These observations will enable us to make use of Theorem~\ref{ubiquity2} in order to prove the following result, which describes the large intersection properties of the set $S^b_{m,n,\psi}$ and thus complements Theorem~\ref{thmBV}.

\begin{thm}\label{grintSchmidt}
Assume that $m+n>2$. Let $b\in\R^m$ and $\psi\in\Psi_{n}$. Then, for any gauge function $h\in\gauge_{m}$ and any nonempty open subset $V$ of $\R^{mn}$,
\[
S^b_{m,n,\psi}\in\grint^{\Id^{m(n-1)}h}(V)
\qquad\Longleftrightarrow\qquad
\sum_{q\in\Z^n\setminus\{0\}}h\left(\frac{\psi(q)}{\|q\|}\right)\|q\|^m=\infty.
\]
\end{thm}

\begin{proof}
Let us first consider the divergence case and assume that $h\prec\Id^m$. Observe that there exists a gauge function $\underline{h}\in\gauge_{m}$ such that $h\prec\underline{h}$ and
\[
\sum_{q\in\Z^n\setminus\{0\}}\underline{h}\left(\frac{\psi(q)}{\|q\|}\right)\|q\|^m=\infty.
\]
Actually, it is possible to build such a gauge function by adapting the methods developed in the proof of~\cite[Theorem~3.5]{Durand:2007lr}. Furthermore, recall that the sets $\mathcal{Q}_{i}$ are defined by~(\ref{defQi}). 
Then,
\[
\infty=\sum_{q\in\Z^n\setminus\{0\}}\underline{h}\left(\frac{\psi(q)}{\|q\|}\right)\|q\|^m\leq\sum_{i=1}^n\sum_{q\in\mathcal{Q}_{i}}\underline{h}\left(\frac{\psi(q)}{\|q\|}\right)\|q\|^m,
\]
so that the sum of $\underline{h}(\psi(q)/\|q\|)\|q\|^m$ over $q\in \mathcal{Q}_{i}$ diverges for some $i\in\intn{1}{n}$. Let $\psi_{i}(q)$ be equal to $\underline{h}(\psi(q)/\|q\|_{2})^{1/m}\|q\|_{2}$ if $q\in\mathcal{Q}_{i}
$ and to zero otherwise. Hence, the series $\sum_{q\in\Z^n}\psi_{i}(q)^m$ diverges. By virtue of Theorem~\ref{thmSchmidt}, the set $S^b_{m,n,\psi_{i}}$ has full Lebesgue measure in $V$. As a consequence, for Lebesgue-
almost every $x\in V$, there are infinitely many $(p,q)\in\Z^m\times \mathcal{Q}_{i}$ such that
\[
\dist_{*}(x,P^b_{(p,q)})<\frac{\psi_{i}(q)}{\|q\|_{2}}=\underline{h}\left(\frac{\psi(q)}{\|q\|_{2}}\right)^{1/m}.
\]
Owing to the fact that~(\ref{existT}) is verified by the subspaces $P^b_{(p,q)}$, for $(p,q)\in\Z^m\times\mathcal{Q}_{i}$, it then follows from Theorem~\ref{ubiquity2} that the set of all $x\in V$ enjoying
\[
\dist_{*}(x,P^b_{(p,q)})<\frac{\psi(q)}{\|q\|_{2}}
\]
for infinitely many $(p,q)\in\Z^m\times \mathcal{Q}_{i}$ belongs to the class $\grint^{\Id^{m(n-1)}h}(V)$. As the $G_{\delta}$-set $S^b_{m,n,\psi}$ contains this last set, it belongs to $\grint^{\Id^{m(n-1)}h}(V)$ as well. 
The result still holds if $h\not\prec\Id^m$. Indeed, in this case, the series $\sum_{q\in\Z^n}\psi(q)^m$ diverges. The set $S^b_{m,n,\psi}$ then has full Lebesgue measure in $V$ due to Theorem~\ref{thmSchmidt}, and 
thus belongs to the class $\grint^{\Id^{m(n-1)}h}(V)$ by Theorem~\ref{grintstable}(\ref{fullleb}).

Let us now consider the convergence case. Observe that there exists a gauge function $\overline{h}\in\gauge_{m}$ such that $\overline{h}\prec h$ and
\[
\sum_{q\in\Z^n\setminus\{0\}}\overline{h}\left(\frac{\psi(q)}{\|q\|}\right)\|q\|^m<\infty.
\]
Again, to build such a function, one may adapt the ideas given in the proof of~\cite[Theorem~3.5]{Durand:2007lr}. Theorem~\ref{thmBV} ensures that the set $S^b_{m,n,\psi}$ has Hausdorff measure zero for the gauge 
function $\Id^{m(n-1)}\overline{h}$. It follows from Theorem~\ref{grintstable}(\ref{relsizelargeint}) that this set cannot belong to the class $\grint^{\Id^{m(n-1)}h}(V)$.
\end{proof}

\begin{rem}
Observe that the gauge functions for which the statement of Theorem~\ref{grintSchmidt} holds are of the form $\Id^{m(n-1)}h$ with $h\in\gauge_{m}$, that is, are those for which Theorem~\ref{thmBV} is valid. In view of 
Remark~\ref{remformgauge} and the relationship between size properties and large intersection properties provided by Theorem~\ref{grintstable}(\ref{relsizelargeint}), it is highly likely that the statement of Theorem~\ref{grintSchmidt} does not hold for the gauge functions that are not of the preceding form.
\end{rem}

\begin{rem}
The hardest part of Theorem~\ref{thmBV}, that is, the divergence part when $h\prec\Id^m$ and $V\neq\emptyset$, may be deduced from Theorem~\ref{grintSchmidt}. Indeed, in this case, if the sum of $h(\psi(q)/\|q\|)\|q
\|^m$ over $q\in\Z^n\setminus\{0\}$ diverges, it is possible to build a gauge function $\underline{h}\in\gauge_{m}$ such that $h\prec\underline{h}$ and the sum of $\underline{h}(\psi(q)/\|q\|)\|q\|^m$ diverges as well. 
Due to Theorem~\ref{grintSchmidt}, the set $S^b_{m,n,\psi}$ then belongs to $\grint^{\Id^{m(n-1)}\underline{h}}(V)$. It finally suffices to apply Theorem~\ref{grintstable}(\ref{relsizelargeint}) to get
\[
\hau^{\Id^{m(n-1)}h}(S^b_{m,n,\psi}\cap V)=\infty=\hau^{\Id^{m(n-1)}h}(V).
\]
\end{rem}

Note that Theorem~\ref{grintSchmidt} directly leads to the part of Theorem~\ref{grintkhintchine} concerning the large intersection properties of the set $K_{m,\phi}$ defined by~(\ref{defKmphi}) where $\phi$ is a 
nonincreasing sequence of positive real numbers converging to zero, because this set can be seen as a particular case of the set $S^b_{m,n,\psi}$.

Using Theorem~\ref{grintSchmidt}, it is also possible to describe the large intersection properties of the set coming into play in Groshev's theorem~\cite{Groshev:1938jk}. Given a nonincreasing sequence $\phi=(\phi(Q))_{Q
\geq 1}$ of positive real numbers converging to zero, this set $\Gamma_{m,n,\phi}$ is formed by the points $(x_{1},\ldots,x_{m})\in (\R^n)^m$ such that
\[
\forall j\in\intn{1}{m} \qquad |q\cdot x_{j}|_{\Z}<\|q\|_{\infty}\,\phi(\|q\|_{\infty})
\]
for infinitely many $q\in\Z^n$. Groshev first studied the size properties of the set $\Gamma_{m,n,\phi}$ by investigating its Lebesgue measure. More recently, Dickinson and Velani~\cite{Dickinson:1997ul} extended 
Groshev's result to the Hausdorff measures associated with fairly general gauge functions. As a complement, the following corollary to Theorem~\ref{grintSchmidt} supplies a description of the large intersection 
properties of $\Gamma_{m,n,\phi}$.

\begin{cor}\label{corGroshev}
Assume that $n>1$. Let $\phi=(\phi(Q))_{Q\geq 1}$ be a nonincreasing sequence of positive real numbers converging to zero. Then, for any gauge function $h\in\gauge_{m}$ and any nonempty open subset $V$ of $
\R^{mn}$,
\[
\Gamma_{m,n,\phi}\in\grint^{\Id^{m(n-1)}h}(V)
\qquad\Longleftrightarrow\qquad
\sum_{Q=1}^\infty h(\phi(Q))Q^{m+n-1}=\infty.
\]
\end{cor}

\begin{proof}
It suffices to apply Theorem~\ref{grintSchmidt} with $b=0$ and $\psi(q)=\|q\|_{\infty}\,\phi(\|q\|_{\infty})$ for all $q\in\Z^n$ and to observe that the number of vectors $q\in\Z^n$ for which $\|q\|_{\infty}=Q$ is 
equivalent to $2^n n\, Q^{n-1}$ as $Q$ tends to infinity.
\end{proof}

\section{Applications to dynamical systems}\label{applicDyn}

\subsection{Perturbation theory for Hamiltonian systems}\label{applicKAM}

The purpose of this subsection is to show how the results obtained in the previous sections may be applied to the perturbation theory for Hamiltonian systems. We shall only give basic recalls on this topic and we refer to~\cite[Chapter~X]{Hairer:2006rm} and~\cite{Poschel:2001rc} for fuller expositions.

The behavior of a general non-dissipative mechanical system with $n$ degrees of freedom may be described through a {\em Hamiltonian system} of differential equations
\[
\dot x_{i}=\frac{\partial H}{\partial y_{i}}, \qquad
\dot y_{i}=-\frac{\partial H}{\partial x_{i}}, \qquad
i\in\intn{1}{n},
\]
where $H:\R^n\times\R^n\to\R$. This system is called {\em integrable} if there exists a {\em canonical} transformation
\[
\begin{array}{rccc}
W:&\R^n\times\T^n &\to& \R^n\times\R^n \\
&(a,\theta) &\mapsto& (x,y)
\end{array}
\]
preserving the symplectic structure such that the Hamiltonian $H\circ W$ (which is simply denoted by $H$ in what follows) does not depend on $\theta$. Here, $\T^n$ denotes the $n$-dimensional torus obtained from $
\R^n$ by identifying the points whose coordinates differ from an integer multiple of $2\pi$. In the {\em action-angle} coordinates $(a,\theta)$, Hamilton's equations then become
\[
\dot\theta_{i}=\frac{\partial H}{\partial a_{i}}, \qquad
\dot a_{i}=-\frac{\partial H}{\partial\theta_{i}}=0, \qquad
i\in\intn{1}{n}
\]
and are clearly solved, for any fixed vector $a_{*}\in\R^n$, by the constant function $a(t)=a_{*}$ and the {\em conditionally periodic} flow $\theta(t)=\theta(0)+t\,\omega(a_{*})$ on the torus $\T^n$ with frequencies $
\omega(a_{*})=(\omega_{1}(a_{*}),\ldots,\omega_{n}(a_{*}))$ given by $\omega_{i}(a_{*})=\partial H/\partial a_{i}(a_{*})$. The flow is periodic if there are integers $q_{1},\ldots,q_{n}$ such that $\omega_{i}(a_{*})/\omega_{i'}
(a_{*})=q_{i}/q_{i'}$ for any $i,i'\in\intn{1}{n}$. Otherwise, the flow is called {\em quasi-periodic}. This occurs in particular when the frequencies are {\em non-resonant}, which means that
\[
\forall q\in\Z^n\setminus\{0\} \qquad q\cdot\omega(a_{*})\neq 0.
\]
Moreover, under this assumption, the trajectory $\{\theta(t),\ t\in\R\}$ is dense in the torus $\T^n$. In any case, the solution curve is winding around the invariant torus $\tor_{a_{*}}=\{a_{*}\}\times\T^n$ with constant 
frequencies $\omega(a_{*})$. Hence, the phase space is {\em foliated} into a $n$-parameter family of invariant tori on which the flow is conditionally periodic.

Integrable Hamiltonian systems raised a large interest because their equations can be solved analytically in the previous manner. The trouble is that, in general, a physical system is not integrable. However, it is often 
possible to view such a system as a perturbation of an integrable approximate one. This observation led to the development of the perturbation theory that we briefly present hereunder. In that context, we may assume that 
the number $n$ of degrees of freedom is at least two, as one degree of freedom systems are always integrable.

Let us consider an invariant torus of an integrable Hamiltonian system, such as for example $\tor_{0}=\{0\}\times\T^n$. It may be shown that this torus is also invariant under the flow of every real-analytic Hamiltonian $H
$ which is not necessarily integrable but for which the linear terms in the Taylor expansion with respect to $a$ at zero do not depend on $\theta$, see~\cite[p.~410]{Hairer:2006rm}. More precisely, this condition amounts 
to the fact that
\begin{equation}\label{HamilTayl}
H(a,\theta)=c+\omega\cdot a+\frac{1}{2} a^T M(a,\theta) a,
\end{equation}
for some real $c$, some vector $\omega\in\R^n$ and some real symmetric $n\times n$-matrix $M(a,\theta)$ analytic in its arguments. For such a Hamiltonian, $\tor_{0}$ is still invariant and the flow on it is conditionally 
periodic with frequencies $\omega$. Let us now consider a perturbation
\[
H(a,\theta)+\eps\,f(a,\theta,\eps)
\]
of such a Hamiltonian, for small $\eps$, by a real-analytic function $f$. Under certain assumptions that we detail below, Kolmogorov (1954) managed to build a near-identity symplectic transformation $(a,\theta)\mapsto 
(\tilde a,\tilde\theta)$ such that the perturbed Hamiltonian in the new variables is also of the form~(\ref{HamilTayl}) with the same $\omega$. It thus admits $\tor_{0}$ as an invariant torus and this torus carries a 
conditionally periodic flow with the same frequencies as the original system. This construction is possible if the angular average
\[
\overline{M}_{0}=\frac{1}{(2\pi)^n}\int_{\T^n} M(0,\theta) \,\dd\theta
\]
is an invertible matrix and if the frequencies $\omega$ satisfy {\em Siegel's Diophantine condition}
\begin{equation}\label{siegelcond}
\forall q\in\Z^n\setminus\{0\} \qquad |q\cdot\omega|\geq\frac{\gamma}{{\|q\|_{1}}^\nu},
\end{equation}
for some positive reals $\gamma$ and $\nu$, where $\|\cdot\|_{1}$ denotes the $\ell^1$-norm. In this case, the frequencies $\omega$ are called {\em strongly non-resonant}. Note that the reals $\gamma$ and $\nu$, together with other parameters, impose a limitation 
on the size $\eps$ of the perturbation for which the construction is possible. Along with its extensions by Arnold (1963) and Moser (1962), Kolmogorov's result forms what is now called the KAM theory.

The existence of strongly non-resonant frequencies is quite obvious, due to the following observation. Given a real $\nu>0$, the set of all frequencies for which the Diophantine condition~(\ref{siegelcond}) holds for some 
$\gamma>0$ is
\[
\Omega_{n,\nu}=\bigcup_{\gamma>0}\uparrow\left\{ \omega\in\R^n \:\biggl|\: |q\cdot\omega|\geq\frac{\gamma}{{\|q\|_{1}}^\nu} \text{ for all }q\in\Z^n\setminus\{0\} \right\}.
\]
If $\nu>n-1$, then it is easy to check that $\Omega_{n,\nu}$ has full Lebesgue measure in $\R^n$. As a result, the set
\[
\Omega_{n}=\bigcup_{\nu>0}\uparrow\Omega_{n,\nu}
\]
formed by the strongly non-resonant frequencies has full Lebesgue measure in $\R^n$. Moreover, the set $\Omega_{n,\nu}$ is empty if $\nu<n-1$, owing to Dirichlet's pigeon-hole principle, and that it has 
Lebesgue measure zero, and Hausdorff dimension $n$, when $\nu=n-1$, see~\cite{Dodson:1986jk} and the references therein.

Let us suppose that $\nu$ is greater than $n-1$. Then, the frequencies for which Siegel's Diophantine condition~(\ref{siegelcond}) does not hold for any $\gamma>0$ form the set
\[
R_{n,\nu}=\R^n\setminus\Omega_{n,\nu}=\bigcap_{\gamma>0}\downarrow\left\{ \omega\in\R^n \:\biggl|\: |q\cdot\omega|<\frac{\gamma}{{\|q\|_{1}}^\nu} \text{ for some }q\in\Z^n\setminus\{0\} \right\}.
\]
Even if it has Lebesgue measure zero, this set is large and omnipresent in various senses. To begin with, $R_{n,\nu}$ is dense $G_{\delta}$-subset of $\R^n$, due to the fact that it contains $\Q^n$. Furthermore, 
M.~Dodson and J.~Vickers~\cite{Dodson:1986jk} proved that
\[
\dim R_{n,\nu}=n-1+\frac{n}{\nu+1},
\]
thereby giving a first description of the size properties of the set $R_{n,\nu}$. Note that the Hausdorff dimension of $R_{n,\nu}$ is always greater than $n-1$ and is therefore almost maximal, that is, equal to $n$, when the 
number $n$ of degrees of freedom is large.

The results of the previous sections lead to the following theorem, which refines the description of the size properties of the set $R_{n,\nu}$ by giving the value of its Hausdorff $\Id^{n-1}h$-measure for every gauge 
function $h\in\gauge_{1}$. On top of that, this theorem shows that $R_{n,\nu}$ is a set with large intersection and it fully describes its large intersection properties.

\begin{thm}\label{thmRnnu}
Let us assume that $n\geq 2$. Let $h\in\gauge_{1}$, let $V$ be a nonempty open subset of $\R^n$ and let $\nu>n-1$. Then,
\[
\left\{\begin{array}{rcl}
\sum_{q} h(q^{-(\nu+1)/n})=\infty & \quad\Longrightarrow\quad & \hau^{\Id^{n-1}h}(R_{n,\nu}\cap V)=\infty \\[1mm]
\sum_{q} h(q^{-(\nu+1)/n})<\infty & \quad\Longrightarrow\quad & \hau^{\Id^{n-1}h}(R_{n,\nu}\cap V)=0.
\end{array}\right.
\]
Moreover,
\[
R_{n,\nu}\in\grint^{\Id^{n-1}h}(V) \quad\Longleftrightarrow\quad \sum\nolimits_{q} h(q^{-(\nu+1)/n})=\infty.
\]
\end{thm}

The proof of this result being quite long, we postpone it to Section~\ref{proofthmRnnu} for the sake of clarity. The frequencies for which Siegel's Diophantine condition~(\ref{siegelcond}) does not hold for any $\gamma>0$ 
and any $\nu>0$, and thus for which Kolmogorov's construction fails, form the set
\[
R_{n}=\R^n\setminus\Omega_{n}=\bigcap_{\nu>0}\downarrow R_{n,\nu}.
\]
As shown by the following result, Theorem~\ref{thmRnnu} leads to a full description of the size and large intersection properties of the set $R_{n}$.

\begin{cor}\label{corRn}
Let us assume that $n\geq 2$. Let $h\in\gauge_{1}$ and let $V$ be a nonempty open subset of $\R^n$. Then,
\[
\left\{\begin{array}{rcl}
[\forall s>0 \quad h(r)\neq{\rm o}(r^s)] & \quad\Longrightarrow\quad & \hau^{\Id^{n-1}h}(R_{n}\cap V)=\infty \\[1mm]
[\exists s>0 \quad h(r)={\rm o}(r^s)] & \quad\Longrightarrow\quad & \hau^{\Id^{n-1}h}(R_{n}\cap V)=0. \\[1mm]
\end{array}\right.
\]
Moreover,
\[
R_{n}\in\grint^{\Id^{n-1}h}(V) \quad\Longleftrightarrow\quad [\forall s>0 \quad h(r)\neq{\rm o}(r^s)].
\]
\end{cor}

\begin{proof}
Let us begin by assuming that $h(r)={\rm o}(r^s)$ for some $s>0$ and let us consider a positive real $\nu$ such that $\nu+1>n/s$. Then, the sum $\sum_{q}q^{-(\nu+1)s/n}$ converges and so does $\sum_{q} h(q^{-(\nu
+1)/n})$. By Theorem~\ref{thmRnnu}, the set $R_{n,\nu}$ has Haudsorff measure zero in $V$ for the gauge $\Id^{n-1}h$. As $R_{n}$ contains this last set, we deduce that $\hau^{\Id^{n-1}h}(R_{n}\cap V)=0$. Furthermore, 
using $\overline{h}=\sqrt{h}$ rather than $h$, we obtain $\hau^{\Id^{n-1}\overline{h}}(R_{n}\cap V)=0$. Hence, $R_{n}\not\in\grint^{\Id^{n-1}h}(V)$ by Theorem~\ref{grintstable}(\ref{relsizelargeint}).

Conversely, let us assume that $h(r)\neq {\rm o}(r^s)$ for all $s>0$. Let $\nu>0$ and suppose that $\sum_{q} h(q^{-(\nu+1)/n})<\infty$. Hence, the function $u\mapsto h(u^{-(\nu+1)/n})$ is integrable at infinity, so that 
for $r>0$ small enough,
\[
\int_{r^{-n/(\nu+1)}/2}^\infty h(u^{-(\nu+1)/n}) \,\dd u\geq\int_{r^{-n/(\nu+1)}/2}^{r^{-n/(\nu+1)}} h(u^{-(\nu+1)/n}) \,\dd u\geq \frac{h(r)}{2r^{n/(\nu+1)}}.
\]
As a result, $h(r)={\rm o}(r^{n/(\nu+1)})$ as $r\to\infty$, which is a contradiction. Thus, the set $R_{n,\nu}$ is in the class $\grint^{\Id^{n-1}h}(V)$ by Theorem~\ref{thmRnnu}. Due to the fact that $\nu\mapsto R_{n,\nu}$ 
is nonincreasing, the intersection defining $R_{n}$ may be written as the intersection over $j\in\N$ of the sets $R_{n,j}$. Therefore, Theorem~\ref{grintstable}(\ref{intersect}) ensures that $R_{n}\in\grint^{\Id^{n-1}h}(V)$. 
Furthermore, it is possible to build a gauge function $\underline{h}\in\gauge_{1}$ such that $h\prec\underline{h}$ and $\underline{h}(r)\neq{\rm o}(r^s)$ for all $s>0$. Using $\underline{h}$ instead of $h$ above, we 
obtain $R_{n}\in\grint^{\Id^{n-1}\underline{h}}(V)$. Theorem~\ref{grintstable}(\ref{relsizelargeint}) finally ensures that $\hau^{\Id^{n-1}h}(R_{n}\cap V)=\infty$.
\end{proof}

Corollary~\ref{corRn} shows that the set $R_{n}$ enjoys a large intersection property in the whole space $\R^n$ for any gauge function of the form $\Id^{n-1}h$, where $h$ grows faster than any power function near zero. 
As a result, the frequencies for which Kolmogorov's construction fails are omnipresent in $\R^n$ in a strong measure theoretic sense. Moreover, a straightforward consequence of Corollary~\ref{corRn} is that the Hausdorff dimension of the set $R_{n}$ is equal to $n-1$. Thus, the frequencies for which Kolmogorov's construction is impossible form a set with almost maximal dimension when the number of degrees of freedom of the system is large.

Finally, let us mention that Siegel's Diophantine condition~(\ref{siegelcond}) also arises in the study of the long-time behavior of symplectic discretizations of integrable Hamiltonian systems (or perturbations of such 
systems). For example, M.~Calvo and E.~Hairer~\cite{Calvo:1995rz} established that the global error of a symplectic numerical integrator on an integrable system grows at most linearly when the frequency at the initial value 
enjoys~(\ref{siegelcond}). Due to Corollary~\ref{corRn}, the set $R_{n}$ of all points for which~(\ref{siegelcond}) does not hold for any $\gamma>0$ and any $\nu>0$ is a set with large intersection and has almost maximal 
Hausdorff dimension in $\R^n$. Thus, the frequencies for which the error growth may not be linear are in some sense prominent in $\R^n$. We refer to~\cite[Chapter~X]{Hairer:2006rm} for other occurrences of Siegel's 
Diophantine condition in the study of symplectic integrators.

\subsection{Rotation number of a homeomorphism of the circle}\label{rotnum}

The study of the continuous orientation preserving homeomorphisms of the circle $\cerc^1=\R/\Z$ yields another application of the results of the previous sections. The rotation number of such a homeomorphism quantifies how much, on average, it moves the points of the circle. It is in fact more convenient to work with lifts of homeomorphisms. Thus, following J.-C.~Yoccoz~\cite{Yoccoz:1992rz}, we shall work with the group $D^0(\cerc^1)$ composed by the continuous homeomorphisms $f$ of $\R$ for which the mapping $x\mapsto f(x)-x$ has period one. For any such function $f\in D^0(\cerc^1)$, the sequence $(f^{\circ q}(x)-x)/q$ converges uniformly in $x$ as $q\to\infty$ to a constant limit $\rho(f)$ called the rotation number of $f$. Here, $f^{\circ q}$ denotes the $q$-fold iteration $f\circ\ldots\circ f$. It is straightforward to check that, for any real $\rho$, the translation $r_{\rho}:x\mapsto x+\rho$ belongs to $D^0(\cerc^1)$ and has rotation number $\rho$.

Alternate definitions of the rotation number and several of its important properties are given in~\cite{Yoccoz:1992rz}. In particular, the rotation number $\rho(f)$ of a given function $f$ in $D^0(\cerc^1)$ is rational if and 
only if the homeomorphism $\tilde f$ of $\cerc^1$ induced by $f$ admits a periodic point. Moreover, if $\rho(f)$ is irrational, then the closure of every orbit of $\tilde f$ is equal to either the whole circle $\cerc^1$ or a 
common Cantor subset ({\em i.e.}~compact, totally disconnected and with no isolated point) of $\cerc^1$. In the first case, $f$ is {\em topologically conjugate} to $r_{\rho(f)}$, that is, there exists a homeomorphism $\phi
\in D^0(\cerc^1)$ such that $\phi\circ f=r_{\rho(f)}\circ\phi$. As shown by Denjoy, this always happens when $f$ is a $C^2$-diffeomorphism and this property is optimal in the sense that, for any $\eps>0$, there exists a 
$C^{2-\eps}$-diffeomorphism with irrational rotation number which is not topologically conjugate to $r_{\rho(f)}$.

Let $f$ be such a diffeomorphism. A further question is that of the smoothness of the conjugacy between $f$ and $r_{\rho(f)}$. The existence of a smooth conjugacy function $\phi$ has been investigated by Moser and 
M.~Herman and is related, in an optimal manner, with the fact that $\rho(f)$ is of {\em Diophantine type} $(K,\sigma)$ for some positive reals $K$ and $\sigma$, which means that
\[
\forall q\in\N \quad \forall p\in\Z \qquad \left| \rho(f)-\frac{p}{q} \right|\geq\frac{K}{q^{\sigma+2}},
\]
see~\cite[Section~2.3]{Yoccoz:1992rz} for details. The results of the preceding sections enable us to study, for any a fixed real $\sigma>0$, the size and large intersection properties of the set $L_{\sigma}$ of all irrational 
numbers that are not of Diophantine type $(K,\sigma)$ for any $K>0$, and thus for which the smoothness results fail. Note that
\[
L_{\sigma}=\bigcap_{K>0}\downarrow\left\{ \rho\in\R\setminus\Q \:\Biggl|\: \left|\rho-\frac{p}{q}\right|<\frac{K}{q^{\sigma+2}}\text{ for some }(p,q)\in\Z\times\N \right\}.
\]
In spite of the fact that it has Lebesgue measure zero, this set may be considered as large in various senses. Indeed, $L_{\sigma}$ is a dense $G_{\delta}$-subset of $\R$. Moreover, V.~Bernik and Dodson~\cite{Bernik:1999qr} proved that the Hausdorff dimension of $L_{\sigma}$ is equal to $2/(2+\sigma)$, thereby being almost maximal in $\R$ when $\sigma$ is small. In addition, as shown by Theorem~\ref{thmLsigma} below, this set 
also enjoys a large intersection property and may thus be seen as omnipresent in $\R$ in a strong measure theoretic sense. Note that this theorem also extends Bernik and Dodson's result by providing a full description of 
the size properties of the set $L_{\sigma}$.

\begin{thm}\label{thmLsigma}
Let $h\in\gauge_{1}$, let $V$ be a nonempty open subset of $\R$ and let $\sigma>0$. Then,
\[
\left\{\begin{array}{rcl}
\sum_{q} h(q^{-(2+\sigma)/2})=\infty & \quad\Longrightarrow\quad & \hau^h(L_{\sigma}\cap V)=\infty \\[1mm]
\sum_{q} h(q^{-(2+\sigma)/2})<\infty & \quad\Longrightarrow\quad & \hau^h(L_{\sigma}\cap V)=0.
\end{array}\right.
\]
Moreover,
\[
L_{\sigma}\in\grint^h(V) \quad\Longleftrightarrow\quad \sum\nolimits_{q} h(q^{-(2+\sigma)/2})=\infty.
\]
\end{thm}

\begin{proof}
To begin with, observe that the set $L_{\sigma}$ is the intersection of $\R\setminus\Q$ with the intersection over $j\in\N$ of the sets
\[
\tilde L_{\sigma,j}=\left\{ \rho\in\R \:\Biggl|\: \left|\rho-\frac{p}{q}\right|<\frac{1}{j\, q^{\sigma+2}}\text{ for infinitely many }(p,q)\in\Z\times\N \right\}.
\]

Let us assume that the sum $\sum_{q} h(q^{-(2+\sigma)/2})$ converges. Then, $\sum_{q} h(q^{-2-\sigma})q$ converges as well and, by Theorem~\ref{grintkhintchine}, the set $\tilde L_{\sigma,1}$ has Hausdorff measure 
zero for the gauge function $h$. As a result, $\hau^h(L_{\sigma}\cap V)=0$. Moreover, there is a gauge function $\overline{h}\in\gauge_{1}$ enjoying $\overline{h}\prec h$ such that the sum $\sum_{q}\overline{h}(q^{-2-
\sigma})q$ converges. Applying what precedes with $\overline{h}$ instead of $h$, we deduce that the set $L_{\sigma}$ has Hausdorff measure zero in $V$ for the gauge function $\overline{h}$. Hence, it cannot belong to 
the class $\grint^h(V)$, owing to Theorem~\ref{grintstable}(\ref{relsizelargeint}).

Let us suppose that the sum $\sum_{q} h(q^{-(2+\sigma)/2})$ diverges. Then, for each $j\in\N$, the sum $\sum_{q} h(1/(j\,q^{\sigma+2}))q$ diverges too. Thanks to Theorem~\ref{grintkhintchine}, the set $\tilde 
L_{\sigma,j}$ belongs to the class $\grint^h(V)$. This is true for any $j\in\N$, so that Theorem~\ref{grintstable}(\ref{intersect}) ensures that this class contains the intersection over $j\in\N$ of the sets $\tilde L_{\sigma,j}$. 
In addition, the set $\R\setminus\Q$ of irrational numbers, being a $G_{\delta}$-subset of $\R$ with full Lebesgue measure, belongs to the class $\grint^h(V)$ as well, owing to Theorem~\ref{grintstable}(\ref{fullleb}). By Theorem~\ref{grintstable}(\ref{intersect}) again, this class finally contains the set $L_{\sigma}$. Furthermore, note that $h\prec\Id$. Thus, there is a gauge function $\underline{h}\in\gauge_{1}$ such that $h\prec\underline{h}$ and the sum $\sum_{q}\underline{h}(q^{-(2+\sigma)/2})$ diverges. Using $\underline{h}$ rather than $h$ above, we obtain $L_{\sigma}\in\grint^{\underline{h}}(V)$. By virtue of Theorem~\ref{grintstable}(\ref{relsizelargeint}), we finally get $\hau^h(L_{\sigma}\cap V)=\infty$.
\end{proof}

To conclude this section, let us mention that the intersection, denoted by $L$, of the sets $L_{\sigma}$ over $\sigma>0$ is the set of Liouville numbers. It is well-known that this set has Hausdorff dimension zero and is a 
dense $G_{\delta}$-subset of $\R$. L.~Olsen~\cite{Olsen:2005fk} established that, for any gauge function $h\in\gauge_{1}$, the set $L$ has Hausdorff measure zero if $h(r)={\rm o}(r^s)$ as $r\to 0$ for some $s>0$ and 
infinite Hausdorff measure in every nonempty open subset of $\R$ otherwise. The following proposition complements this result by describing the large intersection properties of $L$.

\begin{prp}
Let $h\in\gauge_{1}$ and let $V$ be a nonempty open subset of $\R$. Then,
\[
L\in\grint^h(V) \qquad\Longleftrightarrow\qquad [\forall s>0 \quad h(r)\neq{\rm o}(r^s)].
\]
\end{prp}

We refer to~\cite{Durand:2007uq} for a proof of this proposition. A noteworthy consequence of this result is that there are uncountably many ways of writing a given real number as the sum of two Liouville numbers. This is a generalization of a classical result of Erd\H{o}s~\cite{Erdos:1962uq} which states that every real number may be written as the sum of two Liouville numbers.

\section{Proof of Theorem~\ref{ubiquity2}}\label{proofubiquity2}

Let us first assume that $k=0$. Then, for any $i\in I$, there is a point $x_{i}\in\R^d$ such that $P_{i}=\{x_{i}\}$. Note that~(\ref{existT}) is always satisfied, as $T$ may be chosen to be equal to $\R^d$, and that $(x_{i},r_{i})_{i\in I}\in\S_{d}(I)$, because $(P_{i},r_{i})_{i\in I}\in\S_{d}^0(I)$. Moreover, if~(\ref{existhunderline}) holds, then there is a gauge function $\underline{h}\in\gauge_{d}$ with $h\prec\underline{h}$ such that for Lebesgue-almost every $x\in V$, there are infinitely many indices $i\in I$ enjoying $\|x-x_{i}\|<\underline{h}(r_{i})^{1/d}$. Theorem~\ref{ubiquity} implies that $\Fcal_{(P_{i},r_{i})_{i\in I}}$, being equal to $\Fcal_{(x_{i},r_{i})_{i\in I}}$, belongs to the class $\grint^{\underline{h}}(V)$, which is included in the class $\grint^h(V)$ by Proposition~\ref{grintpropbase}.

From now on, let us assume that $k\geq 1$. The proof of Theorem~\ref{ubiquity2} calls upon the following lemma, which can be seen as the analog for net measures of the ``slicing'' lemma of~\cite{Beresnevich:2006lr}, which itself follows from an extension of the 
first part of~\cite[Theorem~10.10]{Mattila:1995fk}. In order to state our slicing lemma, we need to introduce the following notations. For any subset $E$ of $\R^d$ and any $x_{2}\in\R^k$, let
\begin{equation}\label{defEx2}
E_{x_{2}}=\{ x_{1}\in\R^{d-k} \:|\: (x_{1},x_{2})\in E \}.
\end{equation}
Moreover, let $E^*$ denote the set of all $x_{2}\in\R^k$ such that $E_{x_{2}}\neq\emptyset$. Observe that $E$ is the collection of all $(x_{1},x_{2})\in\R^{d-k}\times\R^k$ enjoying $x_{2}\in E^*$ and $x_{1}\in 
E_{x_{2}}$.

\begin{lem}[slicing for net measures]\label{slicing}
Let $h\in\gauge_{d-k}$, let $W$ be an open subset of $\R^d$ and let $E$ denote a subset of $\R^d$. Assume that there exist a real $\kappa>0$ and a subset $W'$ of $W^*$ with full Lebesgue measure such that
\begin{equation}\label{slicinghyp}
\forall x_{2}\in W' \qquad \forall U\subseteq W_{x_{2}}\text{ open} \qquad \netm^h_{\infty}(E_{x_{2}}\cap U)\geq\kappa\,\netm^h_{\infty}(U).
\end{equation}
Then, there exists a real $\kappa'>0$ such that
\[
\forall  U\subseteq W\text{ open} \qquad \netm^{\Id^k h}_{\infty}(E\cap U)\geq\kappa'\,\netm^{\Id^k h}_{\infty}(U).
\]
\end{lem}

\begin{proof}
Let $g=\Id^k h\in\gauge_{d}$. Thanks to Lemmas~8,~9 and~10 in~\cite{Durand:2007uq}, it suffices to prove that there are two reals $\kappa'>0$ and $\rho\in (0,\eps_{g}]$ such that
\begin{equation}\label{slicingeq}
\sum_{p=1}^\infty g(\diam{\lambda_{p}})\geq\kappa'\,g(\diam{\lambda})
\end{equation}
for any $c$-adic cube $\lambda\subseteq W$ with diameter less than $\rho$ and for any sequence $(\lambda_{p})_{p\geq 1}$ in $\Lambda_{c}\cup\{\emptyset\}$ such that $E\cap\lambda\subseteq\bigsqcup_{p}
\lambda_{p}\subseteq\lambda$ (i.e. the sets $\lambda_{p}$ are disjoint, contained in $\lambda$ and cover $E\cap\lambda$). Note that such a cube $\lambda$ is of the form $\lambda=\lambda^{(1)}\times\lambda^{(2)}$, 
where $\lambda^{(1)}$ (resp. $\lambda^{(2)}$) is a $c$-adic cube of $\R^{d-k}$ (resp. $\R^k$). In addition, there is a real $\beta>0$ depending only on the norm $\R^d$ is endowed with such that $\diam{\lambda}=\beta
\,c^{-{\gene{\lambda}_{c}}}$, where $\gene{\lambda}_{c}$ denotes the generation of $\lambda$. Likewise, there is a real $\beta_{1}>0$ such that $\diam{\lambda^{(1)}}=\beta_{1}\,c^{-{\gene{\lambda^{(1)}}_{c}}}$. 
Furthermore, each $\lambda_{p}$ is also of the form $\lambda_{p}^{(1)}\times\lambda_{p}^{(2)}$, where $\lambda_{p}^{(1)}$ and $\lambda_{p}^{(2)}$ are $c$-adic cubes of $\R^{d-k}$ and $\R^k$ respectively, or the 
empty set. When the sets $\lambda_{p}$ and $\lambda_{p}^{(1)}$ are cubes, their diameter may also be expressed in terms of their generation in the previous manner.

In what follows, we choose $\rho$ to be equal to $(1\wedge(\beta/\beta_{1}))\,\eps_{h}$, where $\wedge$ denotes minimum. Note that $\rho\leq\eps_{g}$. As a result, for each integer $p\geq 1$, we have
\begin{eqnarray*}
g(\diam{\lambda_{p}})=h(\diam{\lambda_{p}})\diam{\lambda_{p}}^k &\geq& \left(1\wedge\frac{\beta}{\beta_{1}}\right)^{d-k} h(\diam{\lambda^{(1)}_{p}})\beta^k\leb^k(\lambda^{(2)}_{p})\\
&=& \left(1\wedge\frac{\beta}{\beta_{1}}\right)^{d-k}\beta^k\int_{\lambda^{(2)}} h(\diam{\mu_{p}(x_{2})})\leb^k(\dd x_{2}),
\end{eqnarray*}
where $\mu_{p}(x_{2})$ is equal to $\lambda_{p}^{(1)}$ if $x_{2}\in\lambda_{p}^{(2)}$ and to the empty set otherwise. As $\lambda^{(2)}$ is included in $W^*$ and $W'$ has full Lebesgue measure in $W^*$, we thus obtain
\[
\sum_{p=1}^\infty g(\diam{\lambda_{p}})\geq\left(1\wedge\frac{\beta}{\beta_{1}}\right)^{d-k}\beta^k\int_{\lambda^{(2)}\cap W'}\sum_{p=1}^\infty h(\diam{\mu_{p}(x_{2})})\leb^k(\dd x_{2}).
\]
Observe that, for any $x_{2}\in\lambda^{(2)}\cap W'$, the $c$-adic cubes $\mu_{p}(x_{2})$, for $p\geq 1$, cover the set $E_{x_{2}}\cap\lambda^{(1)}$ and are of diameter less than $\eps_{h}$. As a consequence,
\[
\sum_{p=1}^\infty h(\diam{\mu_{p}(x_{2})})\geq\netm^h_{\infty}(E_{x_{2}}\cap\lambda^{(1)}).
\]
The right-hand side is at least $\netm^h_{\infty}(E_{x_{2}}\cap\interior{\lambda^{(1)}})$, where $\interior{\lambda^{(1)}}$ denotes the interior of $\lambda^{(1)}$. Due to the fact that $\interior{\lambda^{(1)}}$ is an open 
subset of $W_{x_{2}}$, it follows from~(\ref{slicinghyp}) that $\netm^h_{\infty}(E_{x_{2}}\cap\interior{\lambda^{(1)}})$ is at least $\kappa\,\netm^h_{\infty}(\interior{\lambda^{(1)}})$, which is equal to $\kappa
\,h(\diam{\lambda^{(1)}})$ thanks to~\cite[Lemma~9]{Durand:2007uq}. This leads to
\begin{eqnarray*}
\sum_{p=1}^\infty g(\diam{\lambda_{p}}) &\geq& \left(1\wedge\frac{\beta}{\beta_{1}}\right)^{d-k}\beta^k\kappa\,h(\diam{\lambda^{(1)}})\leb^k(\lambda^{(2)}\cap W')\\
&=& \left(1\wedge\frac{\beta}{\beta_{1}}\right)^{d-k}\beta^k\kappa\,h(\diam{\lambda^{(1)}})\leb^k(\lambda^{(2)})\geq\left(\frac{\beta_{1}}{\beta}\wedge\frac{\beta}{\beta_{1}}\right)^{d-k}\kappa\,h(\diam{\lambda})
\diam{\lambda}^k
\end{eqnarray*}
which directly implies~(\ref{slicingeq}).
\end{proof}

We are now able to prove Theorem~\ref{ubiquity2}. To this end, let $h\in\gauge_{d-k}$ and let $V$ be a nonempty open subset of $\R^d$. According to~(\ref{existhunderline}), there exist a gauge function $\underline{h}\in\gauge_{d-k}$ and a subset $V'$ of $V$ with full Lebesgue measure such that, for any point $x\in V'$,~(\ref{ubiquity2eq}) holds for infinitely many indices $i\in I$. Recall that we need to prove that the set $\Fcal_{(P_{i},r_{i})_{i\in I}}$ defined by~(\ref{defFPiri}) belongs to the class $\grint^{\Id^k h}(V)$.

To proceed, let us consider an orthonormal basis $(e_{1},\ldots,e_{d-k})$ of the vector space $\vec{T}$ associated with $T$, a point $a\in T$ and an orthonormal basis $(e_{d-k+1},\ldots,e_{d})$ of the orthogonal complement $\vec{T}^\perp$ of $\vec{T}$. Then, for any $(y_{1},\ldots,y_{d})\in\R^d$, let
\[
\Phi(y_{1},\ldots,y_{d})=a+y_{1}e_{1}+\ldots+y_{d}e_{d}.
\]
Note that there exists a real $\gamma\geq 1$ such that for any $x_{1},x_{1}'\in\R^{d-k}$ and any $x_{2}\in\R^k$,
\begin{equation}\label{Phinorm}
\frac{1}{\gamma}\|x_{1}-x_{1}'\|\leq\|\Phi(x_{1},x_{2})-\Phi(x_{1}',x_{2})\|\leq\gamma\|x_{1}-x_{1}'\|.
\end{equation}

The set $\Phi^{-1}(V')$ has full Lebesgue measure in the open set $W=\Phi^{-1}(V)$ and, using the notations introduced at the beginning of this section, we may deduce from Fubini's theorem that there is a subset $W'$ of 
$W^*$ with full Lebesgue measure such that for every $x_{2}\in W'$ and Lebesgue-almost every $x_{1}\in W_{x_{2}}$, there are infinitely many indices $i\in I$ satisfying
\[
\dist(\Phi(x_{1},x_{2}),P_{i})<\underline{h}(r_{i})^{\frac{1}{d-k}}.
\]

Let $x_{2}\in W'$. Adapting the content of Subsection~4.4.1 in~\cite{Beresnevich:2006lr} to our setting, it is straightforward to check that, owing to~(\ref{existT}), for each $i\in I$, there exists a unique point $z_{i,x_{2}}\in
\R^{d-k}$ enjoying $\Phi(z_{i,x_{2}},x_{2})\in P_{i}$ and that for Lebesgue-almost every $x_{1}\in W_{x_{2}}$, there are infinitely many indices $i\in I$ such that
\[
\|\Phi(x_{1},x_{2})-\Phi(z_{i,x_{2}}, x_{2})\|< C \, \underline{h}(r_{i})^{\frac{1}{d-k}},
\]
where $C$ is the supremum appearing in~(\ref{existT}). Hence, due to~(\ref{Phinorm}), we have
\[
\|x_{1}-z_{i,x_{2}}\|< C\,\gamma\,\underline{h}(r_{i})^{\frac{1}{d-k}}.
\]
As a result, $(z_{i,x_{2}},C\,\gamma\,\underline{h}(r_{i})^{1/(d-k)})_{i\in I}$ is a homogeneous ubiquitous system in $W_{x_{2}}$, see Definition~\ref{ubsys}. Due to~\cite[Proposition~15]{Durand:2007uq}, the family $(z_{i,x_{2}},\underline{h}(r_{i})^{1/(d-k)}/\gamma)_{i\in I}$ is also a homogeneous ubiquitous system in $W_{x_{2}}$, see Remark~\ref{remubsyscst}. Moreover, the fact that $\underline{h}\in\gauge_{d-k}$ clearly implies that
\[
\forall r>0 \qquad \frac{1}{\gamma}\,\underline{h}(r)^{\frac{1}{d-k}}\leq \left(\underline{h}\left(\frac{r}{\gamma}\right)\right)^{\frac{1}{d-k}},
\]
so that $(z_{i,x_{2}},\underline{h}(r_{i}/\gamma)^{1/(d-k)})_{i\in I}$ is a homogeneous ubiquitous system in $W_{x_{2}}$ as well. Owing to Theorem~\ref{ubiquityold}, the set of all $x_{1}\in\R^{d-k}$ such that $\|x_{1}-
z_{i,x_{2}}\|<r_{i}/\gamma$ for infinitely many $i\in I$ belongs to the class $\grint^{\underline{h}}(W_{x_{2}})$, thereby having maximal $\netm^h_{\infty}$-mass in every open subset of $W_{x_{2}}$. Moreover, thanks 
to~(\ref{Phinorm}), this last set is included in the set $(\Phi^{-1}(F))_{x_{2}}$ defined as in~(\ref{defEx2}). Hence,
\[
\forall x_{2}\in W' \qquad \forall U\subseteq W_{x_{2}}\text{ open} \qquad \netm^h_{\infty}((\Phi^{-1}(F))_{x_{2}}\cap U)=\netm^h_{\infty}(U).
\]
Lemma~\ref{slicing} then ensures that $\Phi^{-1}(F)$ has maximal $\netm^{\Id^k h}_{\infty}$-mass in every open subset of $W$. By virtue of~\cite[Lemma~12]{Durand:2007uq}, the set $\Phi^{-1}(F)$ lies in the class $
\grint^{\Id^k h}(W)$ and, owing to Theorem~\ref{grintstable}(\ref{bilip}), the fact that $\Phi$ is bi-Lipschitz finally implies that $F$ belongs to $\grint^{\Id^k h}(V)$ and Theorem~\ref{ubiquity2} follows.

\section{Proof of Theorem~\ref{thmRnnu}}\label{proofthmRnnu}

Before entering the proof of Theorem~\ref{thmRnnu}, let us briefly comment on the summability conditions appearing in the statement. It is easy to check that $\sum_{q}h(q^{-(\nu+1)/n})$ converges if and only if $
\sum_{q} q^{n-1}h(q^{-(\nu+1)})$ does, by comparing these sums with integrals in the usual manner and performing a change of variable. Furthermore, observing that the number of vectors $q\in\Z^n$ for which $\|q\|
_{\infty}=Q$ is equivalent to $2^n n\, Q^{n-1}$ as $Q$ tends to infinity, we deduce that
\begin{equation}\label{equivconv}
\sum_{q=1}^\infty h(q^{-(\nu+1)/n})<\infty \qquad\Longleftrightarrow\qquad \sum_{q\in\Z^n\setminus\{0\}} h({\|q\|_{\infty}}^{-\nu-1})<\infty.
\end{equation}

For the sake of clarity, we split the statement of Theorem~\ref{thmRnnu} into four propositions, namely Propositions~\ref{RnnuconvHau} to \ref{RnnudivHau}, that we now state and establish.

\begin{prp}\label{RnnuconvHau}
Let us assume that $n\geq 2$. Let $h\in\gauge_{1}$, let $V$ be a nonempty open subset of $\R^n$ and let $\nu>n-1$. Then,
\[
\sum\nolimits_{q} h(q^{-(\nu+1)/n})<\infty \qquad\Longrightarrow\qquad \hau^{\Id^{n-1}h}(R_{n,\nu}\cap V)=0.
\]
\end{prp}

\begin{proof}
It suffices to show that, if $\sum_{q} h(q^{-(\nu+1)/n})$ converges, then the Hausdorff measure of $R_{n,\nu}$ for the gauge $\Id^{n-1}h$ is equal to zero. As $R_{n,\nu}$ is stable under the mappings $\omega\mapsto
\lambda\,\omega$, for $\lambda>0$, it is in fact enough to prove that the set $R_{n,\nu}\cap (-1/2,1/2]^n$ has Hausdorff measure zero. To this end, observe that any point $\omega\in R_{n,\nu}$ satisfies $|q\cdot
\omega|<{\|q\|_{\infty}}^{-\nu}$ for infinitely many vectors $q\in\Z^n\setminus\{0\}$. Indeed, a point $\omega\in\R^n$ such that $|q\cdot\omega|\geq {\|q\|_{\infty}}^{-\nu}$ for all $q$ except $q_{1},\ldots,q_{r}$ 
would enjoy $|q\cdot\omega|\geq\alpha{\|q\|_{1}}^{-\nu}$ for all $q$, where $\alpha=\min\{1,|q_{1}\cdot\omega|\,{\|q_{1}\|_{\infty}}^\nu,\ldots,|q_{r}\cdot\omega|\,{\|q_{r}\|_{\infty}}^\nu\}$, and thus could not belong 
to $R_{n,\nu}$. As a consequence,
\[
\forall Q\geq 1 \qquad R_{n,\nu}\cap\left(-\frac{1}{2},\frac{1}{2}\right]^n\subseteq\bigcup_{q\in\Z^n\atop \|q\|_{\infty}\geq Q}\left\{ \omega\in\left(-\frac{1}{2},\frac{1}{2}\right]^n \:\biggl|\: |q\cdot\omega|<\frac{1}{{\|
q\|_{\infty}}^\nu} \right\}.
\]
As pointed out in~\cite[Section~6]{Dodson:1986jk}, each of the sets whose union forms the right-hand side is covered by at most $\beta {\|q\|_{\infty}}^{(n-1)(\nu+1)}$ cubes with diameter $\gamma{\|q\|_{\infty}}^{-
\nu-1}$, where $\beta$ and $\gamma$ are constants greater than one. Along with the fact that $r\mapsto h(r)/r$ is nonincreasing near zero, this implies that for all $\delta>0$ small enough and $Q$ large enough,
\[
\hau^{\Id^{n-1} h}_{\delta}\left(R_{n,\nu}\cap\left(-\frac{1}{2},\frac{1}{2}\right]^n\right)\leq\beta\gamma^n\sum_{q\in\Z^n\atop \|q\|_{\infty}\geq Q}h({\|q\|_{\infty}}^{-\nu-1}).
\]
If the sum $\sum_{q} h(q^{-(\nu+1)/n})$ converges, then the right-hand side tends to zero as $Q$ tends to infinity, by virtue of~(\ref{equivconv}). Letting $\delta$ go to zero, we deduce that the Hausdorff measure of the 
set $R_{n,\nu}\cap (-1/2,1/2]^n$ vanishes.
\end{proof}

\begin{prp}
Let us assume that $n\geq 2$. Let $h\in\gauge_{1}$, let $V$ be a nonempty open subset of $\R^n$ and let $\nu>n-1$. Then,
\[
\sum\nolimits_{q} h(q^{-(\nu+1)/n})<\infty \qquad\Longrightarrow\qquad R_{n,\nu}\not\in\grint^{\Id^{n-1}h}(V).
\]
\end{prp}

\begin{proof}
There is a gauge $\overline{h}\in\gauge_{1}$ such that $\overline{h}\prec h$ and $\sum_{q} \overline{h}(q^{-(\nu+1)/n})$ converges. Employing Proposition~\ref{RnnuconvHau} with $\overline{h}$ rather than $h$, we 
obtain $\hau^{\Id^{n-1}\overline{h}}(R_{n,\nu}\cap V)=0$. Theorem~\ref{grintstable}(\ref{relsizelargeint}) implies that $R_{n,\nu}$ does not belong to $\grint^{\Id^{n-1}h}(V)$.
\end{proof}

\begin{prp}\label{Rnnudivgrint}
Let us assume that $n\geq 2$. Let $h\in\gauge_{1}$, let $V$ be a nonempty open subset of $\R^n$ and let $\nu>n-1$. Then,
\[
\sum\nolimits_{q} h(q^{-(\nu+1)/n})=\infty \qquad\Longrightarrow\qquad R_{n,\nu}\in\grint^{\Id^{n-1}h}(V).
\]
\end{prp}

\begin{proof}
Let $U$ denote a bounded open subset of $\R^{n-1}\times (0,\infty)$ such that the infimum of $x_{n}$ over all $(x_{1},\ldots,x_{n})\in U$ is positive. Then, the mapping $f$ defined by
\[
\forall (x_{1},\ldots,x_{n})\in U \qquad f(x_{1},\ldots,x_{n})=\left(\frac{x_{1}}{x_{n}},\ldots,\frac{x_{n-1}}{x_{n}},x_{n}\right)
\]
is bi-Lipschitz from $U$ onto $W=f(U)$. Note that the coordinates of the elements of $W$ are all bounded by some positive real $\rho$. Furthermore, for any real $\alpha>0$, let us consider the set
\[
\tilde R_{n-1,\nu,\alpha}=\left\{ x\in\R^{n-1} \:\biggl|\: |q\cdot x|_{\Z}<\frac{\alpha}{{\|q\|_{1}}^\nu} \text{ for infinitely many }q\in\Z^{n-1}\setminus\{0\} \right\}
\]
and let $\tilde R_{n-1,\nu}$ denote the intersection over $\alpha>0$ of the sets $\tilde R_{n-1,\nu,\alpha}$. Let $\tilde x=(x_{1},\ldots,x_{n-1})\in\tilde R_{n-1,\nu}$ and let $x_{n}>0$ such that $x=(x_{1},\ldots,x_{n})\in 
W$. Then, for any $\alpha>0$, there exists a vector $\tilde q=(q_{1},\ldots,q_{n-1})\in\Z^{n-1}\setminus\{0\}$ such that $|\tilde q\cdot\tilde x|_{\Z}<\alpha{\|\tilde q\|_{1}}^{-\nu}$. Hence, there is an integer $q_{n}\in\Z
$ such that $|\tilde q\cdot\tilde x+q_{n}|<\alpha{\|\tilde q\|_{1}}^{-\nu}$. Observe that $|q_{n}|\leq |\tilde q\cdot\tilde x|+\alpha{\|\tilde q\|_{1}}^{-\nu}\leq (\rho+\alpha)\|\tilde q\|_{1}$. Therefore,
\[
|\tilde q\cdot (x_{n}\tilde x)+q_{n}x_{n}|<\frac{\alpha x_{n}}{{\|\tilde q\|_{1}}^\nu}\leq \frac{\alpha (1+\rho+\alpha)^\nu \rho}{{\|q\|_{1}}^\nu},
\]
where $q=(q_{1},\ldots,q_{n})\in\Z^n\setminus\{0\}$. It follows that $f^{-1}(x)\in R_{n,\nu}\cap U$. Thus,
\begin{equation}\label{Rnnufinclu}
(\tilde R_{n-1,\nu}\times (0,\infty))\cap W\subseteq f(R_{n,\nu}\cap U).
\end{equation}

Let us now assume that the sum appearing in the statement of the proposition diverges. Then, the gauge $h$ necessarily enjoys $h\prec\Id$ and there exists a gauge $\underline{h}\in\gauge_{1}$ such that $h\prec
\underline{h}$ and the sum $\sum_{q} \underline{h}(q^{-(\nu+1)/n})$ diverges too. Using the same ideas as those leading to~(\ref{equivconv}), it is easy to check that the sum of $\underline{h}({\|q\|_{\infty}}^{-\nu-1})\|q
\|_{\infty}$ over all $q\in\Z^{n-1}\setminus\{0\}$ diverges. Due to the fact that $\underline{h}$ is in $\gauge_{1}$, it follows that
\[
\forall \alpha>0 \qquad \sum_{q\in\Z^{n-1}\setminus\{0\}} \underline{h}\left(\frac{\alpha}{{\|q\|_{1}}^{\nu+1}}\right)\|q\|_{1}=\infty.
\]
Applying Theorem~\ref{grintSchmidt} with $b=0$, $m=1$, $n-1$ instead of $n$ and $\psi(q)=\alpha{\|q\|_{1}}^{-\nu}$, we deduce that each set $\tilde R_{n-1,\nu,\alpha}$ belongs to the class $
\grint^{\Id^{n-2}\underline{h}}(\R^{n-1})$. Note that $\alpha\mapsto\tilde R_{n-1,\nu,\alpha}$ is nonincreasing, so that $\tilde R_{n-1,\nu}$ is also the intersection over all $j\in\N$ of the sets $\tilde R_{n-1,\nu,1/j}$. 
The class $\grint^{\Id^{n-2}\underline{h}}(\R^{n-1})$ thus contains the set $\tilde R_{n-1,\nu}$ by Theorem~\ref{grintstable}(\ref{intersect}). As $h\prec\underline{h}$, this set has maximal $\netm^{\Id^{n-2}h}_{\infty}$-
mass in every open subset of $\R^{n-1}$. By Lemma~\ref{slicing}, there is a real $\kappa'>0$ such that
\[
\netm^{\Id^{n-1}h}_{\infty}((\tilde R_{n-1,\nu}\times (0,\infty))\cap\mathcal{U})\geq\kappa'\,\netm^{\Id^{n-1}h}_{\infty}(\mathcal{U})
\]
for any open subset $\mathcal{U}$ of $\R^{n-1}\times (0,\infty)$. Thanks to~\cite[Lemma~12]{Durand:2007uq}, it follows that the set $\tilde R_{n-1,\nu}\times (0,\infty)$ has maximal $\netm^g_{\infty}$-mass in every 
open subset of $\R^{n-1}\times (0,\infty)$, for any gauge function $g\in\gauge_{n}$ enjoying $g\prec\Id^{n-1}h$. This set thus belongs to the class $\grint^{\Id^{n-1}h}(\R^{n-1}\times (0,\infty))$. Proposition~
\ref{grintpropbase}, along with~(\ref{Rnnufinclu}), then ensures that $f(R_{n,\nu}\cap U)$ belongs to $\grint^{\Id^{n-1}h}(W)$. As $f$ is bi-Lipschitz, Theorem~\ref{grintstable}(\ref{bilip}) implies that $R_{n,\nu}\cap U$ is in 
the class $\grint^{\Id^{n-1}h}(U)$. Consequently, for any gauge $g\in\gauge_{n}$ with $g\prec \Id^{n-1}h$ and any $c$-adic cube $\lambda\subseteq\R^{n-1}\times (0,\infty)$ with diameter less than $\eps_{g}$,
\[
\netm^g_{\infty}(R_{n,\nu}\cap\lambda)\geq\netm^g_{\infty}(R_{n,\nu}\cap\interior{\lambda})=\netm^g_{\infty}(\interior{\lambda})=\netm^g_{\infty}(\lambda),
\]
where the last equality is due to~\cite[Lemma~9]{Durand:2007uq}. Then, using~\cite[Lemma~10]{Durand:2007uq}, we deduce that the set $R_{n,\nu}$ has maximal $\netm^g_{\infty}$-mass in every subset of $
\R^{n-1}\times (0,\infty)$ for any gauge function $g\in\gauge_{n}$ with $g\prec\Id^{n-1}h$. Therefore,
\[
R_{n,\nu}\in\grint^{\Id^{n-1}h}(\R^{n-1}\times (0,\infty)).
\]
Furthermore, $R_{n,\nu}$ is clearly invariant under the bi-Lipschitz mapping $(x_{1},\ldots,x_{n})\mapsto (x_{1},\ldots,-x_{n})$, so that we also have
\[
R_{n,\nu}\in\grint^{\Id^{n-1}h}(\R^{n-1}\times (-\infty,0))
\]
by Theorem~\ref{grintstable}(\ref{bilip}).

Let us now consider a gauge function $g\in\gauge_{n}$ with $g\prec\Id^{n-1}h$ and $c$-adic cube $\lambda\subseteq V$ with diameter less than $\eps_{g}$. The interior $\interior{\lambda}$ of $\lambda$ is an open set 
included in $\R^{n-1}\times (0,\infty)$ or $\R^{n-1}\times (-\infty,0)$. In both cases,
\[
\netm^g_{\infty}(R_{n,\nu}\cap\lambda)\geq\netm^g_{\infty}(R_{n,\nu}\cap\interior{\lambda})=\netm^g_{\infty}(\interior{\lambda})=\netm^g_{\infty}(\lambda),
\]
where the last equality follows from~\cite[Lemma~9]{Durand:2007uq}. Applying~\cite[Lemma~10]{Durand:2007uq}, we deduce that the set $R_{n,\nu}$ has maximal $\netm^g_{\infty}$-mass in every open subset of $V$ 
for every gauge $g\in\gauge_{n}$ with $g\prec\Id^{n-1}h$. Hence, it belongs to the class $\grint^{\Id^{n-1}h}(V)$.
\end{proof}

\begin{prp}\label{RnnudivHau}
Let us assume that $n\geq 2$. Let $h\in\gauge_{1}$, let $V$ be a nonempty open subset of $\R^n$ and let $\nu>n-1$. Then,
\[
\sum\nolimits_{q} h(q^{-(\nu+1)/n})=\infty \qquad\Longrightarrow\qquad \hau^{\Id^{n-1}h}(R_{n,\nu}\cap V)=\infty.
\]
\end{prp}

\begin{proof}
Let us assume that the series appearing in the statement diverges. Observe that the gauge $h$ necessarily enjoys $h\prec\Id$. Thus, there exists a gauge $\underline{h}\in\gauge_{1}$ such that $h\prec\underline{h}$ and 
the sum $\sum_{q} \underline{h}(q^{-(\nu+1)/n})$ diverges too. Owing to Proposition~\ref{Rnnudivgrint}, the set $R_{n,\nu}$ belongs to the class $\grint^{\Id^{n-1}\underline{h}}(V)$. We conclude using Theorem~
\ref{grintstable}(\ref{relsizelargeint}).
\end{proof}


\begin{thebibliography}{99}

\bibitem{Baker:1970jf} A.~Baker and W.M.~Schmidt, Diophantine approximation and {H}ausdorff dimension, {\em Proc. London Math. Soc. (3)} {\bf 21}(3):1--11, 1970.

\bibitem{Beresnevich:1999ys} V.V.~Beresnevich, On approximation of real numbers by real algebraic numbers, {\em Acta Arith.} {\bf 90}(2):97--112, 1999.

\bibitem{Beresnevich:2000fk} V.V.~Beresnevich, Application of the concept of regular systems of points in metric number theory, {\em Vests\=\i~Nats. Akad. Navuk Belarus\={\i} Ser. F\={\i}z.-Mat. Navuk} {\bf 1}:35--39, 2000.

\bibitem{Beresnevich:2002kx} V.V.~Beresnevich, V.I.~Bernik and M.M.~Dodson, Regular systems, ubiquity and {D}iophantine approximation, in {\em A panorama of number theory or The view from Baker's garden}, G.~W{\"u}stholz Ed., Cambridge University Press, Cambridge, pp.~260--279, 2002.

\bibitem{Beresnevich:2006ve} V.V.~Beresnevich, D.~Dickinson and S.L.~Velani, Measure theoretic laws for limsup sets, {\em Mem. Amer. Math. Soc.} {\bf 179}(846):1--91, 2006.

\bibitem{Beresnevich:2005vn} V.V.~Beresnevich and S.L.~Velani, A Mass Transference Principle and the Duffin-Schaeffer conjecture for Hausdorff measures, {\em Ann. of Math. (2)} {\bf 164}(3):971--992, 2006.

\bibitem{Beresnevich:2006lr} V.V.~Beresnevich and S.L.~Velani, Schmidt's theorem, {H}ausdorff measures and slicing, {\em Int. Math. Res. Not.}, 2006, Art. ID 48794, 24 pp.

\bibitem{Bernik:1999qr} V.I.~Bernik and M.M.~Dodson, {\em Metric Diophantine approximation on manifolds}, Cambridge Tracts in Mathematics, vol.~137, Cambridge University Press, Cambridge, 1999.

\bibitem{Besicovitch:1934ly} A.S.~Besicovitch, Sets of fractional dimensions (IV): on rational approximation to real numbers, {\em J. London Math. Soc. (2)} {\bf 9}:126--131, 1934.

\bibitem{Bugeaud:2002fk} Y.~Bugeaud, Approximation par des nombres alg\'ebriques de degr\'e born\'e et dimension de Hausdorff, {\em J. Number Theory} {\bf 96}(1):174--200, 2002.

\bibitem{Bugeaud:2002uq} Y.~Bugeaud, Approximation by algebraic integers and {H}ausdorff dimension, {\em J. London Math. Soc. (2)} {\bf 65}(3):547--559, 2002.

\bibitem{Bugeaud:2004zr} Y.~Bugeaud, An inhomogeneous Jarn\'ik theorem, {\em J. Anal. Math.} {\bf 92}:327--349, 2004.

\bibitem{Calvo:1995rz} M.P.~Calvo and E.~Hairer, Accurate long-term integration of dynamical systems, {\em Appl. Numer. Math.} {\bf 18}(1-3):95--105, 1995.

\bibitem{Dickinson:1997ul} D.~Dickinson and S.L.~Velani, Hausdorff measure and linear forms, {\em J. Reine Angew. Math.} {\bf 490}:1--36, 1997.

\bibitem{Dodson:1986jk} M.M.~Dodson and J.A.G.~Vickers, Exceptional sets in Kolmogorov-Arnol'd-Moser theory, {\em J. Phys. A} {\bf 19}(3):349--374, 1986.

\bibitem{Durand:2007lr} A.~Durand, {\em Propri{\'e}t{\'e}s d'ubiquit{\'e} en analyse multifractale et s{\'e}ries al{\'e}atoires d'ondelettes {\`a} coefficients corr{\'e}l{\'e}s}, PhD thesis, Universit{\'e} Paris 12, 2007.

\bibitem{Durand:2007kx} A.~Durand, Random wavelet series based on a tree-indexed Markov chain, to appear in {\em Comm. Math. Phys.}, arXiv:{\tt 0709.3597}.

\bibitem{Durand:2007uq} A.~Durand, Sets with large intersection and ubiquity, {\em Math. Proc. Cambridge Philos. Soc.} {\bf 144}(1):119--144, 2008.

\bibitem{Durand:2007fk} A.~Durand, Singularity sets of L\'evy processes, {\em Probab. Theory Relat. Fields} (in press), 2008, doi:{\tt 10.1007/s00440-007-0134-6}.

\bibitem{Durand:2006uq} A.~Durand, Ubiquitous systems and metric number theory, {\em Adv. Math.} (in press), 2008, doi:{\tt 10.1016/j.aim.2007.12.008}.

\bibitem{Erdos:1962uq} P.~Erd\H{o}s, Representations of real numbers as sums and products of Liouville numbers, {\em Michigan Math. J.} \textbf{9}:59--60, 1962.

\bibitem{Falconer:1994hx} K.J.~Falconer, Sets with large intersection properties, {\em J. London Math. Soc. (2)} {\bf 49}(2):267--280, 1994.

\bibitem{Falconer:2003oj} K.J.~Falconer, {\em Fractal geometry: Mathematical foundations and applications}, 2nd ed., John Wiley \& Sons Inc., New York, 2003.

\bibitem{Groshev:1938jk} A.V.~Groshev, Un th\'eor\`eme sur les syst\`emes des formes lin\'eaires, {\em Dokl. Akad. Nauk SSSR} {\bf 9}:151--152, 1938.

\bibitem{Hairer:2006rm} E.~Hairer, C.~Lubich and G.~Wanner, {\em Geometric numerical integration}, 2nd ed., Springer Series in Computational Mathematics, vol.~31, Springer-Verlag, Berlin, 2006.

\bibitem{Hardy:1979fk} G.H.~Hardy and E.M.~Wright, {\em An introduction to the theory of numbers}, 5th ed., Oxford University Press, New York, 1979.

\bibitem{Jarnik:1929mf} V.~Jarn\'ik, Diophantischen Approximationen und Hausdorffsches Mass, {\em Mat. Sb.} {\bf 36}:371--381, 1929.

\bibitem{Jarnik:1931qf} V.~Jarn\'ik, \"Uber die simultanen Diophantischen Approximationen, {\em Math. Z.} {\bf 33}(1):505-543, 1931.

\bibitem{Khintchine:1926uq} A.I.~Khintchine, Zur metrischen Theorie der diophantischen Approximationen, {\em Math. Z.} {\bf 24}:706--714, 1926.
  
\bibitem{Mattila:1995fk} P.~Mattila, {\em Geometry of sets and measures in Euclidian spaces}, Cambridge Studies in Advanced Mathematics, vol.~44, Cambridge Univ. Press, 1995.

\bibitem{Olsen:2005fk} L.~Olsen, On the exact {H}ausdorff dimension of the set of {L}iouville numbers, \emph{Manuscripta Math.} {\bf 116}(2):157--172, 2005.

\bibitem{Poschel:2001rc} J.~P\"oschel, A lecture on the classical KAM theorem, {\em Proc. Sympos. Pure Math.} {\bf 69}:707--732, 2001.

\bibitem{Rogers:1970wb} C.A.~Rogers, {\em Hausdorff measures}, Cambridge Univ. Press, Cambridge, 1970.

\bibitem{Schmidt:1964vn} W.M.~Schmidt, Metrical theorems on fractional parts of sequences, {\em Trans. Amer. Math. Soc.} {\bf 110}:493--518, 1964.

\bibitem{Yoccoz:1992rz} J.-C.~Yoccoz, An introduction to small divisors problems, in {\em From number theory to physics}, Springer-Verlag, 1992.

\end{thebibliography}
\end{document}